\documentclass{siamart171218}

\usepackage{amsfonts}
\usepackage{graphicx}
\usepackage{algorithm,algpseudocode}
\usepackage{amssymb}
\usepackage{enumerate}
\usepackage{subfigure}
\usepackage{booktabs}
\usepackage{wrapfig}
\usepackage{tikz}
\usepackage{stix}
\usepackage[skip=2pt,font=scriptsize]{caption}
\ifpdf
  \DeclareGraphicsExtensions{.eps,.pdf,.png,.jpg}
\else
  \DeclareGraphicsExtensions{.eps}
\fi
\usepackage{pgfplots}
\pgfplotsset{compat=1.12}
\let\oldReturn\Return 
\renewcommand{\Return}{\State\oldReturn}

\numberwithin{theorem}{section}
\numberwithin{equation}{section}

\newsiamthm{remark}{Remark}

\newcommand{\TheTitle}{A Hierarchical A-Posteriori Error Estimator\newline for the Reduced Basis Method} 
\newcommand{\TitleShort}{A Hierarchical Error Estimator for the RBM}
\newcommand{\TheAuthors}{S.\ Hain, M.\ Ohlberger, M.\ Radic, K.\ Urban}

\usepackage{amsopn}

\newcommand{\cN}{\mathcal{N}}
\newcommand{\cO}{\mathcal{O}}
\newcommand{\cP}{\mathcal{P}}

\newcommand{\C}{\mathbb{C}}
\newcommand{\N}{\mathbb{N}}
\newcommand{\R}{\mathbb{R}}

\renewcommand{\dagger}{\circ}

\headers{\TitleShort}{\TheAuthors}

\title{{\TheTitle}\thanks{Date: \today.%
	\funding{M.R.\ was supported by the European Union within the EU-MORNet project.}}}

\author{%
  Stefan Hain\thanks{
  	Ulm University, Institute for Numerical Mathematics,  Helmholtzstr.\ 20, 89081 Ulm, 
    \{\email{stefan.hain}, \email{mladjan.radic}, \email{karsten.urban}\}\email{@uni-ulm.de}}
   \and
    Mario Ohlberger\thanks{University of M\"unster, Applied Mathematics,  Einsteinstr.\ 62, 48149 M\"unster,  
  	\email{mario.ohlberger@uni-muenster.de}.}
  \and
  Mladjan Radic\footnotemark[2]
  \and
  Karsten Urban\footnotemark[2]
}

\ifpdf
\hypersetup{
  pdftitle={\TheTitle},
  pdfauthor={\TheAuthors}
}
\fi

\begin{document}
\maketitle

\begin{abstract}
In this contribution we are concerned with tight a posteriori error estimation for projection based model order reduction of  $\inf$-$\sup$ stable parameterized variational problems. In particular, we consider the Reduced Basis Method in a Petrov-Galerkin framework, where the reduced approximation spaces are constructed by the (weak) Greedy algorithm. 
We propose and analyze a hierarchical a posteriori error estimator which evaluates the difference of two reduced approximations of different accuracy. Based on the a priori error analysis of the (weak) Greedy algorithm, it is expected that the hierarchical error estimator is sharp with efficiency index close to one, if the Kolmogorov N-with decays fast for the underlying problem and if a suitable saturation assumption for the reduced approximation is satisfied. We investigate the tightness of the hierarchical a posteriori estimator both from a theoretical and numerical perspective.
For the respective approximation with higher accuracy we study and compare basis enrichment of Lagrange- and Taylor-type reduced bases. Numerical experiments indicate the efficiency for both, the construction of a reduced basis using the hierarchical error estimator in a weak Greedy algorithm, and for tight online certification of reduced approximations. This is particularly relevant in cases where the $\inf$-$\sup$ constant may become small depending on the parameter. In such cases a standard residual-based error estimator -- complemented by the successive constrained method to compute a lower bound of the parameter dependent $\inf$-$\sup$ constant -- may become infeasible.
\end{abstract}



\section{Introduction}

Model order reduction has become a field of great significance, both with respect to solving real world problems and with respect to mathematical research. In this article, we consider the Reduced Basis Method (RBM), which is a well-known projection based model order reduction technique for Parameterized Partial Differential Equations (PPDEs), for 
instance in multi-query and/or real time contexts, \cite{Haasdonk:RB,RozzaRB,QuarteroniRB}. The key idea for the RBM is to construct a problem specific reduced order model -- e.g.\ in a computationally expensive offline phase -- and then use this reduced model to construct an approximation in an online phase extremely fast by solving very low-dimensional Petrov-Galerkin problems. 

A posteriori error estimates play an important role within the RBM, at least for the following reasons: (1) The error estimator is used in a weak Greedy algorithm to construct the reduced model. This is e.g. done by maximizing the error estimator 
over a discrete number of reduced solutions with respect to a finite training set of parameters (`sampling') and to enrich the preliminary reduced basis by the truth solution (`snapshots' ) that corresponds to the worst approximated reduced solution. 
(2) After the online computation of a reduced approximation as a linear combination of the snapshots, an error estimator yields an upper bound for the error and thus certifies the reduced numerical approximation.

This shows that such error estimators need to satisfy a number of conditions: (i) The computation of the error estimator for some given parameter has to be very fast, i.e. with a complexity that only depends on the degrees of freedom of the reduced approximation space (for the basis generation, this allows a large and representative training set; in the online phase, the certification has to be at least as efficient as the computation of the reduced approximation itself); 
(ii) The error estimator has to be tight in order to yield an efficient and reliable estimate of the true error.

So far, the most common approach for constructing such a posteriori RB error estimators is residual-based. This usually involves an efficient computation of (an approximation of) the residual and the inverse of the $\inf$-$\sup$ constant. As for many problems, the $\inf$-$\sup$ constant cannot be computed or estimated in an efficient way, the Successive Constraint Method (SCM) \cite{CHEN20081295,MaxwellSCM,SCM} is used for the calculation of a lower bound. This involves at least two drwabacks, namely the computational complexity of the SCM, in particular if a very good approximation is needed and --related-- the lower bound maybe very small (and thus almost useless for the residual-based error estimator) if the $\inf$-$\sup$ constant is small. Moreover, it has  numerically been observed, that the SCM may not always converge.

Hierarchical error estimators use the difference of two approximations of different order to bound the unknown error. This approach is well-known e.g.\ for ordinary differential equations \cite{MR611953} and adaptive finite elements \cite{BankHier1,MR1392158,MR2950678,Huang2011,MR703179,MR2776914}, just to mention a few. Within the RBM, such an approach has been used to measure the error of the empirical interpolation method (EIM) \cite{EIM,CTU09,DHO12}. We also suppose that such estimators might have been in used in some real-world problems. However, to the very best of our knowledge, we are not aware of an article investigating its use for a posteriori error estimation for RB approximations. 

We investigate two situations: (1) A family of reduced spaces $(X_N)_{i=1,\ldots,N_{\text{max}}}$ is given. Then, we choose $N<M$ and use the difference $\| u_N-u_M\|_X$ of two RB approximations as error estimator in the online phase. We study the performance in particular in those cases, where the $\inf$-$\sup$ constant is small or hard to access numerically. This is e.g.\ the case for the Helmholtz problem, where the $\inf$-$\sup$ constant behaves like $\mu^{-7/2}$, the wave number 
$\mu\in\R^+$ being the parameter. Other examples (that will not be treated here) include transport and wave propagation problems, where one can may construct an optimal reduced space in a possibly costly offline stage but cannot use the residual online, since it cannot be computed efficiently, \cite{RB:Transport,RB:Wave}.  (2) A residual-based error estimator cannot be used at all. In this case, one would like to construct the reduced basis with the aid of the hierarchical error estimator. This, however, is not completely straightforward, since $X_M$ needs to be constructed for given $X_N$. It turns out that a standard greedy procedure may not work in this case. This is the reason why we suggest to use a Taylor-type RB approach for constructing the reduced space of higher accuracy. Numerical experiments are given to demonstrate 
the efficiency of the resulting approach. 

In both cases, (1) and (2), we investigate the effectivity of the hierarchical error estimator, both theoretically and numerically. For the latter purpose, we suggest an offline procedure to determine sharp estimates for the effectivity that can also be used in the online stage.

The remainder of this paper is organized as follows: In Section \ref{sec:preliminaries}, we collect some preliminaries on PPDEs and RBMs. Section \ref{sec:hier_err_est} is devoted to the introduction of the hierarchical error estimator including the analysis and realization. We report on several numerical experiments in Section \ref{sec:NumericalResults} for the standard thermal block problem and the Helmholtz problem in a high frequency regime, i.e. with quite small $\inf$-$\sup$ constants. We mention that our RB hierarchical error estimate has recently been used in the scope of other problems \cite{RB:Transport,2017arXiv170404139F,RB:Wave}.

\section{Preliminaries}
\label{sec:preliminaries}
In this section, we collect the main facts and background material that is used in the sequel.

\subsection{Parameterized Partial Differential Equations (PPDEs)}
Let $\mathcal{P} \subset \mathbb{R}^P$, $P \in \mathbb{N}$, be a compact parameter space. For suitable Hilbert (function) spaces $X$ and $Y$ consider the parameterized variational problem (e.g.\ a PDE): 
\begin{align}
\label{al:ContProbl}
\text{For } \mu \in \mathcal{P} \text{ find } u (\mu) \in X: \; a (u (\mu),v ; \mu) = f (v ; \mu) \qquad \forall v \in Y,
\end{align}
where $a : X \times Y \times \mathcal{P} \rightarrow \mathbb{K}\in\{ \mathbb{R},\mathbb{C}\}$ is a continuous sesquilinear form and $f : Y \times \mathcal{P} \rightarrow \mathbb{K}$ is a given continuous linear form. For ensuring the uniform well-posedness of \eqref{al:ContProbl} for any $\mu \in \mathcal{P}$ one typically assumes that 
\begin{alignat*}{3}
	& \forall \mu \in \mathcal{P}\!: 
	&& \; \sup_{u \in X} \sup_{v \in Y} \frac{\vert a(u,v;\mu) \vert}{\|u\|_X \|v\|_Y} 
	\leq \gamma (\mu) \leq \gamma_{\text{UB}} < \infty , \quad
	&& \text{\emph{(continuity)}}\\
	& \forall \mu \in \mathcal{P}\!: 
	&&  \; \inf_{u \in X} \sup_{v \in Y} \frac{\vert a(u,v;\mu) \vert}{\|u\|_X \|v\|_Y} 
	\geq \beta (\mu) \geq \beta_{\text{LB}}> 0, \quad
	&& \text{\emph{($\inf$-$\sup$ condition)}}\\
	& \forall \mu \in \mathcal{P}\!: 
	&&  \; \inf_{v \in Y} \sup_{u \in X} \frac{\vert a(u,v;\mu) \vert}{\|u\|_X \|v\|_Y}  > 0, \quad
	&& \text{\emph{(surjectivity)}}.
\end{alignat*}
Even though these assumption yield a \emph{uniform} well-posedness (w.r.t.\ the parameter), we note, that particularly $\beta_{\text{LB}}$ may be fairly small, which will be crucial below.

\subsection{The `Truth'}
Next, we require the availability of a detailed or fine discretization in terms of suitable conforming trial and test spaces $X^\mathcal{N} \subset X$ and $Y^\mathcal{N} \subset Y$, where (just for simplicity) $\text{dim}(X^\mathcal{N}) = \text{dim}(Y^\mathcal{N}) =\mathcal{N} < \infty$.  The discretized parameterized problem then reads for any $\mu \in \mathcal{P}$:
\begin{align}\label{al:DiscrProbl}
	\text{Find } u^{\mathcal{N}} (\mu) \in X^{\mathcal{N}}\!\!:\,  
		\, a^{\mathcal{N}} (u^{\mathcal{N}} (\mu),v^{\mathcal{N}} ; \mu) 
			= f^{\mathcal{N}} (v^{\mathcal{N}} ; \mu) 
				\quad \forall v^{\mathcal{N}} \in Y^{\mathcal{N}},
\end{align}
where $a^{\mathcal{N}} : X^{\mathcal{N}} \times Y^{\mathcal{N}} \times \mathcal{P} \rightarrow \mathbb{K}$ and $f^{\mathcal{N}} : Y^{\mathcal{N}} \times \mathcal{P} \rightarrow\mathbb{K}$ are appropriate discrete sesquilinear and linear forms. The discrete sesquilinear and linear forms are continuous with the same constants. To ensure the uniform well-posedness of \eqref{al:DiscrProbl} for every $\mu \in \mathcal{P}$ it is a standard assumption to require
\begin{alignat}{3}
	& \forall \mu \in \mathcal{P}\!\!:
	&&  \; \inf_{u^\mathcal{N} \in X^\mathcal{N}} \sup_{v^\mathcal{N} \in Y^\mathcal{N}} 
\frac{\vert a^\mathcal{N} (u^\mathcal{N},v^\mathcal{N};\mu) \vert }{\|u^\mathcal{N}\|_{X^\mathcal{N}} \|v^\mathcal{N}\|_{Y^\mathcal{N}}} = \beta^\mathcal{N} (\mu) 
	\geq  \beta_{\text{LB}}^\cN > 0. 
	\label{al:discInfSup}
\end{alignat}
Here $\|\cdot\|_{X^\cN}$ and $\|\cdot\|_{Y^\cN}$ may be numerical approximations to $\|\cdot\|_X$ and $\|\cdot\|_Y$, respectively, but may also be discrete norms (such as for discontinuous Galerkin -dG- methods). Such a detailed discretization can e.g.\ arise from Finite Element, Finite Volume, dG or Spectral Element discretizations.

It is a standard assumption that this detailed discretization is sufficiently fine so that the error $\| u(\mu)-u^\cN(\mu)\|_X$ is negligible, which is the reason why $u^\cN(\mu)$ is often called the `truth'. In particular, we assume here that $X^\cN$ and $Y^\cN$ are the same for all parameters, but mention that adaptive discretizations may also be used 
(cf. \cite{WaveletRB,HDO11}).

\subsection{The Reduced Basis Method (RBM)}
\label{sec:RBM}
We briefly recall the main ingredients of the Reduced Basis Method (RBM) which we need here and refer e.g.\ to \cite{Haasdonk:RB,RozzaRB,QuarteroniRB} for more details. 
The aim of the RBM is to determine a highly reduced model of size $N \ll \mathcal{N}$ in terms of reduced trial and test spaces $X_N\subset X^\cN$, $Y_N\subset Y^\cN$.
Such a reduced model is typically determined in an offline phase, which might be computationally costly. This is done by selecting certain parameters $S_N:=\{\mu_1, \ldots , \mu_N\}$, computing the corresponding (truth) snapshots $\xi_i := u^{\mathcal{N}}(\mu_i)$, $i=1,\ldots, N$, and setting $X_N := \text{span} \{  \xi_1 , \ldots, \xi_N \}$, $N \ll \mathcal{N}$. The basis may be orthonormalized for stability reasons. 

The choice of the snapshot parameter set $S_N$ is usually based upon an efficiently computable a posteriori error estimator $\Delta_N(\mu)$ which is then maximized in a greedy manner over a finite training set $\mathcal{P}_{\text{train}} \subset \mathcal{P}$. This approach is called \emph{weak greedy}. Sometimes, the error is used instead of an error estimator, which is then termed as \emph{strong greedy}. Other approaches such as nonlinear optimization of an error estimator have also been investigated, e.g.\ \cite{Oli}.

In order to ensure well-posedness of the reduced problem, namely:
\begin{align}
\label{al:RBApprox}
\text{For } \mu \in \mathcal{P} \text{ find } u_N (\mu) \in X_N\!\!: \; a^\mathcal{N} (u_N (\mu),v_N ; \mu) = f^{\mathcal{N}} (v_N;\mu) \quad \forall v_N \in Y_N,
\end{align}
the spaces $X_N$ and $Y_N$ have to be chosen such that
\begin{equation}\label{RB:infsup}
	\inf_{w_N \in X_N} \sup_{v_N \in Y_N} 
		\frac{\vert a^\mathcal{N} (w_N,v_N;\mu) \vert }{\|w_N\|_{X^\cN} \|v_N\|_{Y^\cN}} 
		=: \beta_N^\cN (\mu) 
		\geq \beta_{\text{LB}}^\cN > 0,
		\qquad \mu\in\cP.
\end{equation}
Let $u_N ( \mu) = \sum_{i=1}^N u_{i,N} (\mu) \xi_i$ be the desired expansion of the RB approximation. It is easily seen that the unknown coefficient vector $\mathbf{u}_N (\mu) =(u_{i,N} (\mu) )_{i=1}^N$ arises from solving a linear system of equations $\mathbb{A}_N (\mu) \mathbf{u}_N (\mu) = \mathbb{F}_N (\mu)$, where $(\mathbb{A}_N (\mu))_{i,j}:=a^{\mathcal{N}}(\xi_i,\eta_j;\mu)$, $(\mathbb{F}_N (\mu))_{j}:=f^{\mathcal{N}} (\eta_j;\mu)$, and $Y_N := \text{span} \{ \eta_1 , \ldots , \eta_N \}$ is the reduced test space. Typically, $\mathbb{A}_N (\mu)$ is a dense matrix so that the reduced approximation can be computed with $\mathcal{O} (N^3)$ operations. This complexity is independent of the truth dimension $\cN$, which is the reason to call it \emph{online efficient}. 
In order to setup the linear system in an online efficient manner, it is usually assumed that sesquilinear and linear forms are separable w.r.t.\ the parameter, i.e,,
\begin{align}
	\label{eq:affinedecomp}
		a^{\mathcal{N}}(w,v;\mu) 
		&= \sum_{q=1}^{Q^a} \vartheta_q^a (\mu) a_q^{\mathcal{N}}(w, v), 
		\qquad && \mu \in \mathcal{P}, w \in X^{\mathcal{N}}, v \in Y^{\mathcal{N}}, \\
	f^{\mathcal{N}}(v;\mu) 
		& = \sum_{q=1}^{Q^f} \vartheta_q^f (\mu) f_q^{\mathcal{N}}(v),  
		\qquad && \mu \in \mathcal{P}, v \in Y^{\mathcal{N}}. 
\end{align}
Sometimes \eqref{eq:affinedecomp} is also called \emph{affine decomposition}. If \eqref{eq:affinedecomp} is not satisfied, the empirical interpolation method can be used to construct an affine approximation (see e.g. \cite{EIM}). Using \eqref{eq:affinedecomp}, one can precompute parameter-independent quantities in the offline stage allowing for an online efficient setup of the linear system. In fact, the parameter-independent matrices and vectors $(\mathbb{A}_N^q)_{j,i} := a_q^\mathcal{N} (\xi_i,\eta_j)$,  $i, j=1,\ldots,N, \ q=1,\ldots,Q^{a}$ and $(\mathbb{F}_N^q)_{j} := f_q^\mathcal{N} (\eta_j)$, $j=1,\ldots,N, \ q=1,\ldots,Q^{f}$, can be computed offline and stored once. Then, for a given new parameter $\mu\in\cP$
\begin{alignat*}{3}
	\mathbb{A}_N (\mu) = \sum_{q=1}^{Q^{a}}  \vartheta_q^{a} (\mu) \mathbb{A}_N^q, \qquad
	\mathbb{F}_N (\mu) = \sum_{q=1}^{Q^{f}}  \vartheta_q^{f} (\mu) \mathbb{F}_N^q,
\end{alignat*}
which is of complexity $\cO(Q^{a}N^2)$ and $\cO(Q^{f}N)$, respectively. As the complexity does not dependent on $\cN$,
it is \emph{online efficient}.

The best possible rate of convergence for the error is given by the decay of the \emph{Kolmogorov $N$-width}
\begin{equation}\label{eq:nwidth}
	d_N(\cP) := \inf_{\dim(X_N)=N, X_N\subset X} \sup_{\mu\in\cP} \inf_{v_N\in X_N} \| u(\mu)-v_N\|_X.
\end{equation}
It is known that $d_N(\cP)$ decays fast (even exponentially) for several PPDEs as $N\to\infty$ with smooth dependence of the solution on the parameter (see e.g. \cite{OR16}).

\subsection{The residual based a-posteriori error estimator}
As already mentioned above, an online efficient error estimator $\Delta_N(\mu)$ is often used within a weak greedy procedure to determine the snapshot index set $S_N$. Moreover, such a $\Delta_N(\mu)$ is used for online certification by computing an upper bound for the error induced by the RB approximation $u_N(\mu)$.  In this paper, we will consider two examples for such a $\Delta_N(\mu)$. For the subsequent analysis, we will consider
$$
	e_N(\mu) := \| u(\mu) - u_N(\mu)\|_{X},
		\qquad
	e^\cN_N(\mu) := \| u^\cN(\mu) - u_N(\mu)\|_{X^\cN}
$$
which will be termed \emph{exact error} and \emph{truth error}, respectively. Also other error quantities or functions of the error can be considered using adjoint methods. It is fairly standard to use the (truth) \emph{residual} $R_N^{\mathcal{N}}(\cdot;\mu)\in (Y^\cN)'$ defined as
$$
	R_N^{\mathcal{N}} (w ; \mu) := f^{\mathcal{N}} (w;\mu) - a^\mathcal{N} (u_N(\mu),w ; \mu) 
		= a^\mathcal{N} (e^\cN_N(\mu),w ; \mu),
		\quad w\in Y^\cN,
$$
to define the residual based a-posteriori RB error estimator as follows
\begin{align*}
	\Delta_N^{\text{Std}}(\mu) 
		:= \frac{\| R_N^{\mathcal{N}} (\cdot ; \mu) \|_{(Y^{\mathcal{N}})^\prime } }{\beta^{\mathcal{N}} (\mu) }, 
\end{align*}
which we will call \emph{standard RB error estimator} in the sequel. It should be noted that the (truth) residual also admits an affine decomposition and can thus in fact be computed online efficient. The involved (truth) $\inf$-$\sup$ constant $\beta^{\mathcal{N}} (\mu)$ can only be determined exactly in very specific cases. Usually, a lower bound $\beta_{\text{LB}}^{\mathcal{N}} (\mu)$ is computed for example by the \emph{Successive Constraint Method} (SCM), \cite{MaxwellSCM,HesthavenEFI,SCM}. However, even though the SCM is online efficient, the quantitative performance may be a severe problem in realtime applications, in particular if a good approximation of  $\beta^{\mathcal{N}} (\mu)$ is required (which is the case, e.g., if $\beta^{\mathcal{N}} (\mu)$ is small).

The relation of the truth error and the residual is well-known and easily seen
\begin{align}
\label{al:ResError}
	\frac{1}{\gamma^{\mathcal{N}} (\mu)} \| R_N^{\mathcal{N}} (\cdot ; \mu) \|_{(Y^{\mathcal{N}})^\prime } 
		\leq
		 \| e_N^{\mathcal{N}} (\mu) \|_{X^\mathcal{N}} 
		 \leq 
 	\frac{1}{\beta^{\mathcal{N}} (\mu)} \| R_N^{\mathcal{N}} (\cdot ; \mu) \|_{(Y^{\mathcal{N}})^\prime }. 
\end{align}
Note, that this relation is w.r.t.\ the truth error, not w.r.t.\ the exact error \cite{WaveletRB,Ohlberger:true,MR3431132,MR3318670,MR3460105}. Of course, one can replace $\beta^{\mathcal{N}} (\mu)$ and $\gamma^{\mathcal{N}} (\mu)$ in \eqref{al:ResError} by lower and upper bounds $\beta_{\text{LB}}>0$, $\gamma_{\text{UB}}<\infty$, respectively, even though these bounds may be numerically infeasible. Under the assumptions of the previous sections, it has been proven that weak greedy algorithms exhibit the same rate of convergence as $d_N(\mathcal{P})$ if there exists rigorous lower and upper bounds for the error, like \eqref{al:ResError}, see \cite{ConvergenceRB,MR2877366}. Roughly speaking the RBM works well for a PPDE if $d_N(\cP)$ decays sufficiently fast as $N$ grows.

\section{A Hierarchical Error Estimator}
\label{sec:hier_err_est}
In this section, we introduce the hierarchical error estimator. To this end, let $X_N \subsetneq X_M \subset X^{\mathcal{N}}$, where $\text{dim}(X_M) = M > N = \text{dim}(X_N)$, and $u_N(\mu) \in X_N$, $u_M(\mu) \in X_M$, respectively. Then, we define the hierarchical error estimator by
\begin{eqnarray}\label{eq:HierErrEst}
	\Delta_{N,M}(\mu) := \Vert u_{M}(\mu) - u_N(\mu) \Vert_{X^{\mathcal{N}}},
\end{eqnarray}

\subsection{Error Analysis}
\label{sec:ErrorAnalysis}

The analysis of hierarchical error estimators is pretty standard in various applications for ODEs or PDEs. Due to the specific framework of parameter-dependent problems, we detail it here. We indicate two approaches.

\subsubsection*{Asymptotic analysis} 
Using triangle inequality, we get by \eqref{al:ResError} and \eqref{RB:infsup}
\begin{eqnarray*}
	\Vert u^{\cN}(\mu) - u_N(\mu) \Vert_{X^{\mathcal{N}}}
	&\le& \Vert u^{\cN}(\mu) - u_M(\mu) \Vert_{X^{\mathcal{N}}}
		+ \Vert u_{M}(\mu) - u_N(\mu) \Vert_{X^{\mathcal{N}}} \\
	&=& \Vert u^{\cN}(\mu) - u_M(\mu) \Vert_{X^{\mathcal{N}}}
		+ \Delta_{N,M}(\mu) \\
	&\le& \frac{1}{\beta^{\mathcal{N}} (\mu)} \| R_M^{\mathcal{N}} (\cdot ; \mu) \|_{(Y^{\mathcal{N}})^\prime } 
		+ \Delta_{N,M}(\mu) 
	= \Delta_M^{\text{Std}}(\mu) + \Delta_{N,M}(\mu). 
\end{eqnarray*}
Now, we recall from \cite{ConvergenceRB} that one can construct $X_M$ in such a way that $\Delta_M^{\text{Std}}(\mu)\to 0$ as $M\to\infty$ for every $\mu \in \mathcal{P}$ provided that the Kolmogorov $M$-width decays, i.e., this is a term of higher order. This means that for any $N$ and $\varepsilon>0$, we can choose an $M=M(\varepsilon)>N$ such that
$$
	\Vert u^{\cN}(\mu) - u_N(\mu) \Vert_{X^{\mathcal{N}}} \le \varepsilon +  \Delta_{N,M}(\mu).
$$
Alternatively, we can choose $M$ such that $\Delta_M^{\text{Std}}(\mu)\le \varepsilon  \Delta_{N,M}(\mu)$ yielding that
$$
	\Vert u^{\cN}(\mu) - u_N(\mu) \Vert_{X^{\mathcal{N}}} \le  {(1+ \varepsilon)} \cdot \Delta_{N,M}(\mu).
$$
If, however, the assumption $\Delta_M^{\text{Std}}(\mu)\le \varepsilon  \Delta_{N,M}(\mu)$ is only satisfied on 
a training set $\mathcal{P}_{\text{train}} \subset \mathcal{P}$, there might exists parameters $\mu \in \mathcal{P}\setminus \mathcal{P}_{\text{train}}$ with $\Delta_{N,M}(\mu) = 0$ for all $M$, but $\Vert u^{\cN}(\mu) - u_N(\mu) \Vert_{X^{\mathcal{N}}} \neq 0$. This may happen if $X_M$ does not converge to $X^\cN$, which motivates a further assumption.

\subsubsection*{Saturation assumption}
A way to analyze hierarchical error estimates is by showing or assuming a guaranteed error decay, typically called \emph{saturation property}, see e.g.\ \cite{BankHier1,Huang2011,WohlmuthHier}. In order to formulate it, we recall that the reduced spaces $X_N := \text{span} \{  \xi_1 , \ldots, \xi_N \}$, $N \ll \mathcal{N}$ are formed by snapshots $\xi_i := u^{\mathcal{N}}(\mu_i)$, $i=1,\ldots,N$. Consider now a second reduced basis space $X_M$ with $\text{dim}(X_M) = M > N = \text{dim}(X_N)$. Then, we say that $X_N$ and $X_M$ satisfy the \emph{saturation property}, if there exists 
a constant $\Theta_{N,M}^{\mathcal{N}} \in (0,1)$, s.t.  
\begin{eqnarray}\label{eq:saturation}
	\Vert u^{\mathcal{N}}(\mu) - u_{M}(\mu) \Vert_{X^{\mathcal{N}}} 
	\leq \Theta_{N,M}^{\mathcal{N}} \cdot \Vert u^{\mathcal{N}}(\mu) - u_{N}(\mu) \Vert_{X^{\mathcal{N}}}
\end{eqnarray}
holds for all $\mu \in \mathcal{P}$. We will show a numerical procedure to validate this assumption below. At this point we do \emph{not} specify the particular construction of $X_M$, see \S \ref{sec:RBBasisGenNew} below. 
Then, following standard lines, we can easily prove the following estimates.
\begin{proposition}\label{thm:HierEst1}
	If \eqref{eq:saturation} holds, then 
	\begin{equation}\label{eq:HierErrEstBound}
		\frac{\Delta_{N,M}(\mu)}{1+\Theta_{N,M}^{\mathcal{N}}}  
		\le \Vert u^{\mathcal{N}}(\mu) - u_N(\mu) \Vert_{X^{\mathcal{N}}}
		\le \frac{\Delta_{N,M}(\mu)}{1-\Theta_{N,M}^{\mathcal{N}}} =: \Delta_{N,M}^\mathrm{Hier}(\mu).
	\end{equation}
\end{proposition}
\begin{proof} 
For $\mu \in \mathcal{P}$ with $\Vert u^{\mathcal{N}}(\mu) - u_N(\mu) \Vert_{X^{\mathcal{N}}} = 0$  the inequalities are obviously fulfilled. If $\Vert u^{\mathcal{N}}(\mu) - u_N(\mu) \Vert_{X^{\mathcal{N}}} \not= 0$, we use the reverse triangle inequality and the saturation assumption to obtain
\begin{align*}\label{eq:derivation_hier_err_est}
\frac{\Vert u_{M}(\mu) - u_N(\mu) \Vert_{X^{\mathcal{N}}}}{\Vert u^{\mathcal{N}}(\mu) - u_N(\mu) \Vert_{X^{\mathcal{N}}}} 
	&\geq \frac{ \Vert u^{\mathcal{N}}(\mu) - u_N(\mu) \Vert_{X^{\mathcal{N}}} - \Vert u^{\mathcal{N}}(\mu) - u_{M}(\mu) \Vert_{X^{\mathcal{N}}}}{\Vert u^{\mathcal{N}}(\mu) - u_N(\mu) \Vert_{X^{\mathcal{N}}}} \\
	&= 1 - \frac{\Vert u^{\mathcal{N}}(\mu) - u_{M}(\mu) \Vert_{X^{\mathcal{N}}}}{\Vert u^{\mathcal{N}}(\mu) - u_N(\mu) \Vert_{X^{\mathcal{N}}}} \geq 1 - \sup_{\mu \in \mathcal{P}}{\frac{\Vert u^{\mathcal{N}}(\mu) - u_{M}(\mu) \Vert_{X^{\mathcal{N}}}}{\Vert u^{\mathcal{N}}(\mu) - u_N(\mu) \Vert_{X^{\mathcal{N}}}}} \\
	&\geq 1 - \Theta_{N,M}^{\mathcal{N}}, 
\end{align*}
which proves the upper bound. The lower bound is proven by triangle inequality and saturation.
\end{proof}

\begin{remark}
	With a slight abuse of terminology, we sometimes call both $\Delta_{N,M}$ and $\Delta_{N,M}^\mathrm{Hier}$ ``hierarchical error estimator''. Strictly speaking, only $\Delta_{N,M}^\mathrm{Hier}$ is an  upper bound bound for the error, whereas $\Delta_{N,M}$ requires the multiplicative constant $(1 - \Theta_{N,M}^{\mathcal{N}})^{-1}$ in order to be an upper bound.
\end{remark}

For the effectivity 
\begin{align}\label{eq:effectivityfactor}
	\eta_{N,M}^{\mathcal{N}}(\mu) := \frac{\Delta_{N,M}(\mu)}{(1 - \Theta_{N,M}^{\mathcal{N}}) \Vert u^{\mathcal{N}}(\mu) - u_N(\mu) \Vert_{X^{\mathcal{N}}}}
\end{align}
we obviously get that
\begin{align}\label{eq:Sharpness}
	1 \leq \eta_{N,M}^{\mathcal{N}}(\mu) \leq \frac{1+\Theta_{N,M}^{\mathcal{N}}}{1 - \Theta_{N,M}^{\mathcal{N}}}.
\end{align}
The closer $\Theta_{N,M}^{\mathcal{N}}$ is to zero, the better is the effectivity.

\subsection{Realization}
The hierarchical error estimator can be computed online-efficient as we are going to show now. In fact, let
$$
	u_N(\mu) = \sum_{i=1}^N \alpha^N_i(\mu)\, \xi_i,
	\qquad
	u_M(\mu) = \sum_{i=1}^M \alpha^M_i(\mu)\, \xi_i,
$$
be the expansions of the reduced basis approximations (in general $\alpha^N_i(\mu)\not= \alpha^M_i(\mu)$ even for $1\le i\le N$). Then, setting $\alpha^N_i(\mu):=0$ for $i=N+1,\ldots ,M$, we get
\begin{eqnarray*}
	\Delta_{N,M}(\mu)^2
	&=& \left\| \sum_{i=1}^M (\alpha^N_i(\mu) - \alpha^M_i(\mu))\, \xi_i\right\|_{X^\cN}^2 \\
	&=& \sum_{i,j=1}^M (\alpha^N_i(\mu) - \alpha^M_i(\mu)) (\alpha^N_j(\mu) - \alpha^M_j(\mu))\, (\xi_i, \xi_j)_{X^\cN}.
\end{eqnarray*}
Since the values $(\xi_i, \xi_j)_{X^\cN}$ (the entries of the Gramian matrix) can be precomputed and stored in the offline stage, the computation of $\Delta_{N,M}(\mu)$ requires $\cO(M^2)$ operations independent of $\cN$, i.e., online efficient. Of course, we have the well-known \emph{square root effect}, since the above reasoning yields $\Delta_{N,M}(\mu)^2$ so that we loose half of the accuracy by taking the square root. This, however, is exactly the same for the standard estimator and there are suggestions how to deal with it
(see e.g. \cite{Buhr20144094}).

\subsection{Offline approximation of $\Theta_{N,M}^{\mathcal{N}}$}
\label{sec:ConstThetaNew}
The main challenges for using the hierarchical error estimator are (i) the choice of an appropriate $M$ and (ii) the determination of the multiplicative constant $\rho$ with $e^\cN_N(\mu) \le \rho \Delta_{N,M}(\mu)$ for all $\mu\in\cP$. Obviously, both issues are linked. In the case using the saturation assumption, we have that $\rho = (1-\Theta_{N,M}^{\mathcal{N}})^{-1}$, so that we start describing an offline procedure to approximate the saturation constant.

To this end, we use a result on nonlinear parametrized programming problems.
\begin{theorem} \cite{Dinkelbach}\label{thm:dinkelbach}
Let $\cP\subset \mathbb{R}^P$ be compact and connected, $f,g: \cP\to\R$ continuous such that $g(\mu) > 0$ for all $\mu\in\cP$. Setting $F(q):=  \max_{\mu\in\cP}{\{f(\mu) - q \cdot g(\mu)\}}$, $q\in \R$, it holds $q_0 := \max_{\mu\in\cP}{\frac{f(\mu)}{g(\mu)}}$ if and only if $F(q_0) = 0$.
\end{theorem}

We apply this result for the functions $f(\mu) := \Vert u^{\mathcal{N}}(\mu) - u_M(\mu) \Vert_{X^{\mathcal{N}}}$ and $g(\mu) := \Vert u^{\mathcal{N}}(\mu) - u_N(\mu) \Vert_{X^{\mathcal{N}}}$. Due to the requirement $g(\mu) > 0$ for all $\mu\in\cP$, we decompose the parameter space $\mathcal{P}$ in compact subsets $\mathcal{P}_i$ in such a way, that on each subset the denominator is non-vanishing. In view of \eqref{al:ResError} this means here that $\Vert R_N^{\mathcal{N}}(\cdot ; \mu) \Vert_{(Y^{\mathcal{N}})'} \neq 0$. Then, we proceed as follows: for fixed dimension $N$ and for $\mathcal{P}_i$ we solve the nonlinear problem
\begin{align}\label{eq:rootfindingproblem}
\Theta_{N,M,i}^{\mathcal{N}} := \text{arg} \: \min_{q \in \mathbb{R}_{\geq 0}}{\vert F_i(q) \vert} \quad \text{with} \quad  F_i(q):= \max_{\mu \in \mathcal{P}_i}\{f(\mu) - q \cdot g(\mu)\}
\end{align}
and define $\Theta_{N,M}^{\mathcal{N}} := \max_{i}{\Theta_{N,M,i}^{\mathcal{N}}}$. For each $i$, we construct an iteration $\theta_i^{(k)}$, $k=0,1,2,\ldots$, for which we need a good starting value $\theta_i^{(0)}$. Since
\begin{align*}
	\frac{\beta^{\mathcal{N}}(\mu)}{\gamma^{\mathcal{N}}(\mu)} \cdot \frac{\| R_M^{\mathcal{N}} (\cdot ; \mu) \|_{(Y^{\mathcal{N}})^\prime }}{\| R_N^{\mathcal{N}} (\cdot ; \mu) \|_{(Y^{\mathcal{N}})^\prime }} 
	\leq \frac{\Vert u^{\mathcal{N}}(\mu) - u_M(\mu) \Vert_{X^{\mathcal{N}}}}{\Vert u^{\mathcal{N}}(\mu) - u_N(\mu) \Vert_{X^{\mathcal{N}}}} 
	\leq \frac{\gamma^{\mathcal{N}}(\mu)}{\beta^{\mathcal{N}}(\mu)} \cdot \frac{\| R_M^{\mathcal{N}} (\cdot ; \mu) \|_{(Y^{\mathcal{N}})^\prime }}{\| R_N^{\mathcal{N}} (\cdot ; \mu) \|_{(Y^{\mathcal{N}})^\prime }}, 
\end{align*}
we use the following approximation as initial guess 
\begin{align*}
	\Theta_{N,M,i}^{\mathcal{N}} 
	:= \max_{\mu \in \mathcal{P}_i}{\frac{\Vert u^{\mathcal{N}}(\mu) - u_M(\mu) \Vert_{X^{\mathcal{N}}}}{\Vert u^{\mathcal{N}}(\mu) - u_N(\mu) \Vert_{X^{\mathcal{N}}}}} 
	\approx \max_{\mu \in \mathcal{P}_i}{ \frac{\| R_M^{\mathcal{N}} (\cdot ; \mu) \|_{(Y^{\mathcal{N}})^\prime }}{\| R_N^{\mathcal{N}} (\cdot ; \mu) \|_{(Y^{\mathcal{N}})^\prime }}} 
		=: \theta_i^{(0)},
\end{align*}
which is reasonable provided that $\min\limits_{\mu \in \mathcal{P}_i}{\frac{\beta_{\text{LB}}^{\mathcal{N}}(\mu)}{\gamma_{\text{UB}}^{\mathcal{N}}(\mu)}} \approx \max\limits_{\mu \in \mathcal{P}_i}{\frac{\gamma_{\text{LB}}^{\mathcal{N}}(\mu)}{\beta_{\text{UB}}^{\mathcal{N}}(\mu)}}$. This results in the (offline) Algorithm \ref{algo:GetConstant}. If this algorithm terminates with some $\Theta_{N,M}^{\mathcal{N}}<1$, the saturation property is in fact valid.
\goodbreak
\begin{algorithm}[!htb]
\caption{Computing $\Theta_{N,M}^{\mathcal{N}}$}
\label{algo:GetConstant} 
\begin{algorithmic}[1]
\State{Choose $\texttt{tol}>0$, fix $N \in \mathbb{N}$, choose $L\in\N$ compact subsets $\mathcal{P}_i$, $1 \leq i \leq L$}
\For{$i = 1 : L$}
\State{$f(\mu) := \Vert u^{\mathcal{N}}(\mu) - u_M(\mu) \Vert_{X^{\mathcal{N}}}$, %
	$g(\mu) := \Vert u^{\mathcal{N}}(\mu) - u_N(\mu) \Vert_{X^{\mathcal{N}}}$}
\State{$F_i(q) := \max\limits_{\mu \in \mathcal{P}_i}\{f(\mu) - q \cdot g(\mu)\}$}
\State{$k:=0$}
\State{$\theta_i^{(0)} := \max\limits_{\mu \in \mathcal{P}_i}{ \frac{\| R_M^{\mathcal{N}} (\cdot ; \mu) \|_{(Y^{\mathcal{N}})^\prime }}{\| R_N^{\mathcal{N}} (\cdot ; \mu) \|_{(Y^{\mathcal{N}})^\prime }}}$}  \label{line:thetanull}
\While{$\vert F_i(\theta_i^{(k)}) \vert \geq \texttt{tol} $}
	\State {iteratie nonlinear problem $F_i(q) = 0$ $\leadsto \theta_i^{(k+1)}$}
	\State {$k\to k+1$}
\EndWhile
\State{$\Theta_{N,M,i}^{\mathcal{N}}:=\theta_i^{(k)}$}
\EndFor
\Return $\Theta_{N,M}^{\mathcal{N}} := \max\limits_{i = 1, \ldots , L}{\Theta_{N,M,i}^{\mathcal{N}}}$
\end{algorithmic}
\end{algorithm}

At least quantitatively, the following might be more efficient instead of line \ref{line:thetanull}:\nobreak
\begin{algorithmic}[1]
\makeatletter
\setcounter{ALG@line}{5}
\makeatother
\State{ $\mu_i^* := \text{arg}\max\limits_{\mu \in \mathcal{P}_i}{\frac{\Vert R_M^{\mathcal{N}}(\cdot; \mu) \Vert_{(Y^{\mathcal{N}})'}}{\Vert R_N^{\mathcal{N}}(\cdot; \mu) \Vert_{(Y^{\mathcal{N}})'}}}$, 
	$\theta_i^{(0)} := \frac{\Vert u^{\mathcal{N}}(\mu_i^*) - u_N^{(d)}(\mu_i^*) \Vert_{X^{\mathcal{N}}}}{\Vert u^{\mathcal{N}}(\mu_i^*) - u_N^{(0)}(\mu_i^*) \Vert_{X^{\mathcal{N}}}}$}
\end{algorithmic}

\subsection{Reduced Basis Generation}
\label{sec:RBBasisGenNew}
So far, we assumed that  $X_N$ and $X_M$ are given, e.g.\ by a strong greedy method in an offline phase without using the hierarchical error estimator. One could also think of using the hierarchical part $\Delta_{N,M}(\mu)$ for this purpose. This, however, is at least not straightforward since one needs \emph{both} $N$ and $M$ for the error estimator, where $M$ has to be sufficiently large from the beginning. It would be a straightforward approach to start with $N=1$, $M=2$ for some parameters $\mu_1\not=\mu_2$. Maximizing $\Delta_{1,2}(\mu)$ over a training set would yield $\mu_3$ and we would set $N=2$, $M=3$, $S_3=\{ \mu_1,\mu_2,\mu_3\}$, etc. However, it can relatively easy be seen that this approach does not necessarily converge as snapshots may be selected repeatedly. Hence, we suggest a different approach.

Starting with $X_N$, the saturation property \eqref{eq:saturation} is always valid as long as the Kolmogorov $N$-width decays and the reduced basis has been constructed with a weak greedy algorithm. However, this only means that for each RB space $X_N$ there exists an appropriate RB space $X_M$, s.t.\ \eqref{eq:saturation}  is satisfied -- one is left with the question how to construct such a space $X_M$. We suggest to use the Taylor-RB method. If the solution $u(\mu)$ depends smoothly on the parameter $\mu$, we can add derivatives of the snapshots w.r.t.\  the respective parameter to the basis, i.e., for $X_N=\text{span}\{ u(\mu_1), \ldots , u(\mu_N)\}$ we set
$$
	X_M \kern-1pt:= \kern-1pt\text{span}\bigg\{  
		u(\mu_n),  
		 \frac{\partial^k}{\partial \mu_i^k} u(\mu_n):\, k=1,\ldots , K_n,\, i=1,\ldots, P,\, n=1,\ldots, N  
		 \bigg\}
$$
for appropriately chosen $K_n\in\N_0$. This means that $M = \sum_{n=1}^N (1+K_n\cdot P)$. 
It is well-known that these \emph{Taylor snapshots} $u^{(k)}_i(\mu):=\frac{\partial^k}{\partial \mu_i^k} u(\mu)$ can easily be computed recursively by solving the following linear variational problem (see e.g.\ \cite{QuarteroniRB})
\begin{align}\label{eq:ModifiedRBDerivatives} 
	a(u^{(k)}_i(\mu) , v ; \mu) 
	= \frac{\partial^k}{\partial \mu_i^k} f(v;\mu) 
	- \sum_{m=1}^k{\begin{pmatrix} k \\ m \end{pmatrix} \frac{\partial^m }{\partial \mu_i^m}
	a( u^{(k-m)}_i(\mu),v;\mu )}.
\end{align}
In general, the partial derivatives appearing in \eqref{eq:ModifiedRBDerivatives} are G\^ateaux derivatives. However, if the affine decomposition \eqref{eq:affinedecomp} holds, one just needs the derivatives of the involved functions $\theta_q^a, \theta_{q'}^{f}\!: \mathcal{P} \rightarrow \mathbb{R}$ in the classical sense. In this case one can ensure by standard arguments that for each $N$  there exists some  $M>N$, s.t.\ the results of \S\ref{sec:ErrorAnalysis} hold, provided that the solution is real-analytic with respect to $\mu$. Finally, for stability reasons we orthonormalize the Taylor snapshots by a POD. The corresponding method is summarized in Algorithm \ref{algo:WeakGreedyHi}.

\begin{algorithm}
\caption{(Weak) Greedy with Hierarchical Error Estimator}
\label{algo:WeakGreedyHi}
\begin{algorithmic}[1] 
	\State{Choose $\text{tol}>0$, $N_{\max}$, $\mathcal{P}_{\text{train}} \subset \mathcal{P}$, $\mu_1\in\cP$}
	\State{ $S_1 := \{\mu_1\}$,  
		$\Xi_1^{(0)} := \{\xi_1 := u^{\mathcal{N}}(\mu_1) \}$}
	\For{$N = 1, \ldots , N_{\text{max}}$}
		\State{$k = 1$}
	\Repeat\label{line:5}
		\State {$\check\Xi_{N}^{(k)} := \{ u^{(k)}_i(\mu_N):\, i=1,\ldots, P\}$ computed by \eqref{eq:ModifiedRBDerivatives}}
		\State{$\Xi_{N}^{(k)} := \text{ORTHONORMALIZE}(\Xi_N^{(k-1)}, \check\Xi_{N}^{(k)})$}
		\State{Set $X_N := \text{span} (S_N)$, $X_M := \text{span}(\Xi_N^{(k)})$
			 compute $\Theta_{N,M}^{\mathcal{N}}$ by Algorithm \ref{algo:GetConstant}} \label{line:ComputeTheta1}
		\State{$k \leftarrow k + 1$}
	\Until{$\Theta_{N,M}^{\mathcal{N}} < 1$}\label{line:10}
	\State $K_N:=k$
	\If{$\max_{\mu \in \mathcal{P}_{\text{train}}}{\Delta_{N,M}(\mu)} < \text{tol}$}
		\State STOP
	\Else
		\State{$\mu_{N+1} := \text{arg} \max\limits_{\mu \in \mathcal{P}_{\text{train}}}{\Delta_{N,M}(\mu)}$}
		\State{$S_{N+1} := S_N\cup \{\mu_{N+1}\}$, $\Xi_{N+1}^{(0)} := \Xi_N^{(0)} \cup \{ \xi_{N+1} := u^{\mathcal{N}}(\mu_{N+1}) \}$}
	\EndIf
	\EndFor
	\Return $S_N$, $X_N:= \text{span} (S_N)$ and $X_M := \text{span}(\Xi_N^{(K)})$, $\Theta_{N,M}^{\mathcal{N}}$
\end{algorithmic}
\end{algorithm} 

\begin{remark}
It can be expected (and we have indeed confirmed this by several numerical experiments) that the saturation property \eqref{eq:saturation} can be realized by decomposing the parameter space similar to \cite{MR2746617} (there called ``hp-RBM''). In addition to Algorithm \ref{algo:WeakGreedyHi} we have realized such an hp-RBM approach by modifying lines \ref{line:5} to \ref{line:10}. We observed fast convergence.
\end{remark}

\section{Numerical Results}
\label{sec:NumericalResults}
We investigate the quantitative performance of the RB hierarchical error estimator and focus on the sharpness and asymptotically correctness of \eqref{eq:HierErrEst}. In particular, we want to investigate
\begin{enumerate}
	\item How is the performance of $\Delta_{N,M}^\mathrm{Hier}$ as compared to $\Delta_N^\mathrm{Std}$?
	\item How does the performance depend on the availability of a sharp lower $\inf$-$\sup$ bound?
	\item Since $\Delta_{N,M}$ is an upper bound for the error up to some multiplicative constant depending on $M$, what is a reasonable choice for that constant?
	\item What is a good choice for $X_M$?
\end{enumerate}
For that purpose, we report on experiments for two test problems. All experiments have been performed on iMac 2009 equipped with an Intel Core 2 Duo 3.06 GHz processor and 8 GB 1067 MHz DDR3 RAM.

The first example, the so-called `thermal block' from \cite{Patera}, is a well-known benchmark problem for the RBM. In this case, the behavior of the $\inf$-$\sup$/coercivity constant is known and the performance of the SCM is very good such that $\Delta_N^\mathrm{Std}$ is expected to yield good results. We expect that $\Delta_{N,M}^\mathrm{Hier}$ should be less sharp for general $X_M$ and we are particularly interested in a quantitative comparison. The second example is the Helmholtz problem which has also been investigated in the RB-context in \cite{Haasdonk:RB}. In this case, it is known that the $\inf$-$\sup$ constant has a poor behavior for large parameters \cite{Melenk} and -moreover- the computation of a decent approximation using the SCM is quite costly. Hence, this should be a good benchmark test for the hierarchical error estimator. 

For the basis generation, we use both the strong and the weak greedy algorithm based upon $\Delta_{N,M}$ and $\Delta_N^\mathrm{Std}$ w.r.t.\ the same training set $\mathcal{P}_{\mathrm{train}}$. For $\Delta_{N,M}^\mathrm{Hier}$, we compare constructions of $X_M$ using a Taylor and a Lagrange basis. 

\begin{remark}
\begin{enumerate}
\item For simplicity we compute $\Theta_{M,N}$ over a training set, i.e.
\begin{align*}
	\Theta_{N,M}^{\mathcal{N}, \mathrm{train}} 
	&:= \max_{\mu \in \mathcal{P}_{\mathrm{train}}} \frac{\Vert u^{\mathcal{N}}(\mu) - u_{M}(\mu) \Vert_{X^{\mathcal{N}}}}{\Vert u^{\mathcal{N}}(\mu) - u_N(\mu) \Vert_{X^{\mathcal{N}}}},
\end{align*}
instead of solving the nonlinear problem \eqref{eq:rootfindingproblem}.
\item Although all problems considered here are stationary,  the hierarchical error estimator can also be applied to instationary problems e.g.\ by using a space-time formulation, \cite{SpaceTimeSchwabStevenson,SpaceTimeUrbanPatera}. 
\end{enumerate}
\end{remark}


\subsection{Thermal-Block (see \cite{Patera})}
\label{sec:ThermalBlock}
Let $\Omega := (0,1)^2$, divided into $B_1 \times B_2$ rectangular subblocks $\Omega_i \subset \Omega$, s.t. $\overline{\Omega} = \bigcup_{i=1}^{B_1B_2}{\overline{\Omega_i}}$. Let $ \mu \in \mathcal{P} \subset \mathbb{R}^2$ and $\alpha(x;\mu) := \mu_j\, \chi_{\Omega_i}(x)$ for $j \in \{1,2\}$, $1 \leq i \leq B_1\cdot B_2$, $\mu = (\mu_1,\mu_2) \in \mathcal{P}$, where $j=1$ if and only if $i$ is odd. We consider stationary heat conduction 
\begin{align*}\label{eq:ThermalBlock}
	\begin{split}
	- \nabla \cdot \left(\alpha(x;\mu) \; \nabla u(x;\mu) \right) 
		&= 0, \qquad \qquad \quad x \in \Omega, \\
	u(x;\mu) 
		&= 0, \qquad \qquad \quad x \in \Gamma_D := \{(x,1)^T \in \mathbb{R}^2 : 0 \leq x \leq 1\}, \\
	\alpha(x;\mu) \frac{\partial u}{\partial n}(x) 
		&= g_N(x;\mu), \qquad x \in \Gamma_N := \partial \Omega \backslash \Gamma_D.
	\end{split}
\end{align*}
Here, we choose $B_1 = B_2 = 3$ (see figure below) and set
\begin{align*}
	g_N(x;\mu) := \begin{cases} 
		1, \qquad \mathrm{on} \ \{(x,0)^T \in \mathbb{R}^2 : 0 \leq x \leq 1\},  \\ 
		0, \qquad \mathrm{on} \ \{(0,y)^T \in \mathbb{R}^2 : 0 \leq y \leq 1\} \cup \{(1,y)^T \in \mathbb{R}^2 : 0 \leq y \leq 1\}.
		\end{cases}
\end{align*}
Then, we have a coercive problem with identical trial and test space $X=Y:=H_D^1(\Omega) := \{ v \in H^1(\Omega) : v \vert_{\Gamma_D} = 0\}$ as well as bilnear and linear forms defined as\newline
\begin{minipage}{0.55\textwidth}
\begin{align*}
a ( u ,v;\mu) = & \sum_{i=1}^{\lceil B_1B_2 \rceil/2} {\color{red} \mu_1} \int_{\Omega_{2i-1}} \nabla u\cdot \nabla v ~ dx \\
& + \sum_{i=1}^{\lceil B_1B_2 \rceil / 2 - 1} {\color{blue} \mu_2} \int_{\Omega_{2i}} \nabla u\cdot \nabla v ~ dx,  \\
f(v;\mu) = & \int_{\Gamma_N} v \; dx.
\end{align*}
\end{minipage}
\begin{minipage}{0.39\textwidth}
\begin{tikzpicture}
\draw (0,0) -- (3,0);
\draw (0,0) -- (0,3);
\draw (3,0) -- (3,3);
\draw (0,3) -- (3,3);

\draw (0,1) -- (3,1);
\draw (0,2) -- (3,2);

\draw (1,0) -- (1,3);
\draw (2,0) -- (2,3);

\draw node at (0.5,0.5) {\color{red}$\Omega_1$};
\draw node at (0.5,1.5) {\color{blue}$\Omega_2$};
\draw node at (0.5,2.5) {\color{red}$\Omega_3$};

\draw node at (1.5,0.5) {\color{blue}$\Omega_4$};
\draw node at (1.5,1.5) {\color{red}$\Omega_5$};
\draw node at (1.5,2.5) {\color{blue}$\Omega_6$};

\draw node at (2.5,0.5) {\color{red}$\Omega_7$};
\draw node at (2.5,1.5) {\color{blue}$\Omega_8$};
\draw node at (2.5,2.5) {\color{red}$\Omega_9$};

\draw node at (-0.35,1.5) {$\Gamma_N$};
\draw node at (3.35,1.5) {$\Gamma_N$};
\draw node at (1.5,-0.25) {$\Gamma_N$};
\draw node at (1.5,3.25) {$\Gamma_D$};

\end{tikzpicture}
\end{minipage}

For the truth discretization, we used piecewise linear finite elements  with a total number of 11.881 degrees of freedom. Further, we choose two different parameter spaces, namely 
$$
	\mathcal{P}^{(1)} = [0.5,1]^2, 
	\qquad
	\mathcal{P}^{(2)} = [0.02,1]^2,
	\qquad
	|\mathcal{P}_{\mathrm{train}}^{(1)}| 
	= |\mathcal{P}_{\mathrm{train}}^{(2)}| = 10.201. 
$$
For the error plots, the discrete coercivity constant  (replacing the $\inf$-$\sup$ constant) was determined as the smallest eigenvalue of a generalized eigenvalue problem. For the online CPU-time for computing $\Delta_N^\mathrm{Std}$, we used the SCM.

For the thermal block problem, the solution depends only mildly on the parameter. Hence, the SCM converges after only $3$  steps to numerical precision, even on the larger parameter space $\mathcal{P}^{(2)}$. Therefore, we expect that $\Delta_N^\mathrm{Std}$ is quite sharp, which is confirmed by our experiments.  
Starting with the smaller parameter set $\mathcal{P}^{(1)}$, we also found $\Delta_{N,M}^\mathrm{Hier}$ to be quite sharp even for $M=N+1$. We omit the corresponding figures since $\Delta_N^\mathrm{Std}$ and $\Delta_{N,N+1}^\mathrm{Hier}$ turned out to be almost indistinguishable. Hence, we consider the larger parameter set $\mathcal{P}^{(2)}\supset \mathcal{P}^{(1)}$. The results are displayed in Figure \ref{fig:Thermal_M2_strong} using the strong greedy
and in Figure \ref{fig:Thermal_M2_weak} for the weak greedy with $\Delta_N^\mathrm{Std}$ for the sampling. We do not see a significant difference between the different sampling methods to create the reduced basis spaces. In addition, we also did the parameter sampling by the hierarchical error estimator. We omit the corresponding figures since the results are quite similar to Figures \ref{fig:Thermal_M2_strong} and \ref{fig:Thermal_M2_weak}.

In both figures, we use 100 test parameters and plot the true error in red solid lines. The dashed blue lines correspond to the average value of $\Delta_N^\mathrm{Std}(\mu)$ for these 100 test parameters. Finally, the dotted black lines indicate the average values of $\Delta_{N,M}^\mathrm{Hier}$ for $M\in\{ N+1, N+2\}$ using a Taylor-based construction with $K_n=1$ and $K_n=2$, respectively. We see a significant improvement for $M=N+2$ and almost no difference to $\Delta_N^\mathrm{Std}$.

In the tables next to the figures, we monitor the constants $\Theta_{N,M}$ for both choices. As expected, the value $\Theta_{N,M}$ significantly improves for $M=N+2$. 
However, in all cases the constant is below $1$ and we can easily deduce online heuristics. 
\begin{figure}[!htb]
	\caption{\label{fig:Thermal_M2_strong}Thermal-Block, 
		$\mathcal{P}^{(2)} = [0.02,1]^2$, 
		strong greedy sampling.
		Average error over test set of parameters. 
		Red, solid: true error; blue, dashed: residual error estimator; 
		black, dotted: hierarchical error estimator, $M\in\{N+1,N+2\}$.}
	\subfigure[Strong greedy, $M = N + 1$.]{
		\begin{tikzpicture}[scale=0.45]
		\begin{axis}[
			width=0.75\textwidth,
			xlabel={$N$},
			ymode=log,
			cycle list name=color list,
			unbounded coords=jump,
			legend pos=outer north east,
			legend cell align=left,
			legend style={draw=none,font=\tiny},
			clip=false,xtick=data,
			y=0.37cm, 
			]
			\addplot+[line width=1.5pt] table[y=err] 
				{./ThermalBlock_Basisgen_Strong_Strong_M=N+1.txt};
			\addplot+[line width=1.5pt,dashed] table[y=std] 
				{./ThermalBlock_Basisgen_Strong_Strong_M=N+1.txt};
			\addplot+[line width=1.5pt,dotted] table[y=hier] 
				{./ThermalBlock_Basisgen_Strong_Strong_M=N+1.txt};
		\end{axis}
		\end{tikzpicture}
		\label{fig:Thermal_M2_strong_N+1}
	}
	\scriptsize{
	\begin{tabular}[b]{cc}
	\toprule
	$N$ & $\Theta_{N,N+1}$ \\ 
	\midrule
	$1$ & $0.9736$ \\ 
	$2$ & $0.9156$ \\ 
	$3$ & $0.9677$ \\ 
	$4$ & $0.9141$ \\ 
	$5$ & $0.2970$ \\ 
	$6$ & $0.1716$ \\ 
	$7$ & $0.8133$ \\ 
	$8$ & $0.6927$ \\ 
	$9$ & $0.8543$ \\ 
	$10$ & $0.2969$ \\ 
	\bottomrule
	\end{tabular} 
	}
	\subfigure[Strong greedy, $M = N + 2$.]{
		\begin{tikzpicture}[scale=0.45]
		\begin{axis}[
			width=0.75\textwidth,
			xlabel={$N$},
			ymode=log,
			cycle list name=color list,
			unbounded coords=jump,
			legend pos=outer north east,
			legend cell align=left,
			legend style={draw=none,font=\tiny},
			clip=false,xtick=data,
			y=0.45cm, 
			]
			\addplot+[line width=1.5pt] table[y=err] 
				{./ThermalBlock_Basisgen_Strong_Strong_M=N+2.txt};
			\addplot+[line width=1.5pt,dashed] table[y=std] 
				{./ThermalBlock_Basisgen_Strong_Strong_M=N+2.txt};
			\addplot+[line width=1.5pt,dotted] table[y=hier] 
				{./ThermalBlock_Basisgen_Strong_Strong_M=N+2.txt};
		\end{axis}
		\end{tikzpicture}
		\label{fig:Thermal_M2_strong_N+2}
		}
	\scriptsize{
	\begin{tabular}[b]{cc}
	\toprule
	$N$ & $\Theta_{N,N+2}$ \\ 
	\midrule
	$1$ & $0.5987$ \\ 
	$2$ & $0.4890$ \\ 
	$3$ & $0.6214$ \\ 
	$4$ & $0.2715$ \\ 
	$5$ & $0.0186$ \\ 
	$6$ & $0.1382$ \\ 
	$7$ & $0.1693$ \\ 
	$8$ & $0.1413$ \\ 
	$9$ & $0.1031$ \\ 
	$10$ & $0.0133$ \\ 
	\bottomrule
	\end{tabular} 
	}
\end{figure}
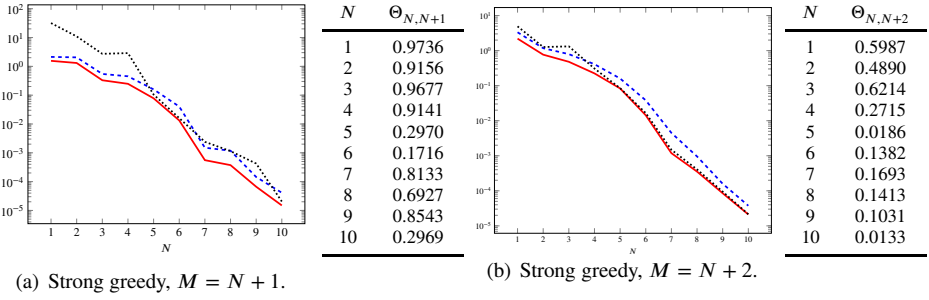 
%
%
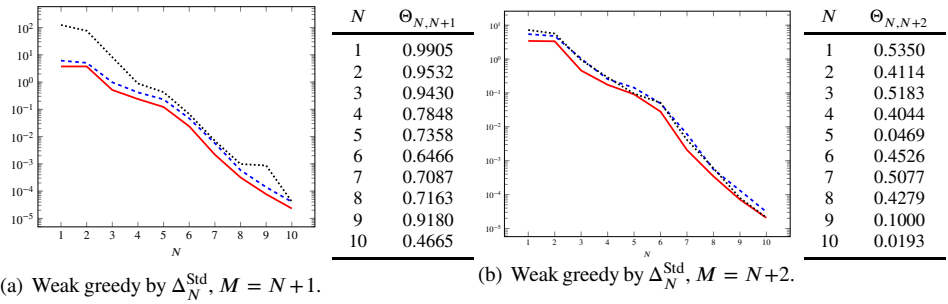
\begin{figure}[!htb]
	\caption{\label{fig:Thermal_M2_weak}Thermal-Block, 
	$\mathcal{P}^{(2)} = [0.02,1]^2$, 
	weak greedy with standard error estimator.
		Average error over test set of parameters. 
		Red, solid: true error; blue, dashed: residual error estimator; 
		black, dotted: hierarchical error estimator, $M\in\{ N+1, N+2\}$.}
	\subfigure[Weak greedy by $\Delta_N^{\text{Std}}$, $M = N + 1$.]{
		\begin{tikzpicture}[scale=0.45]
		\begin{axis}[
			width=0.75\textwidth,
			xlabel={$N$},
			ymode=log,
			cycle list name=color list,
			unbounded coords=jump,
			legend pos=outer north east,
			legend cell align=left,
			legend style={draw=none,font=\tiny},
			clip=false,xtick=data,
			y=0.35cm, 
			]
			\addplot+[line width=1.5pt] table[y=err] 
				{./ThermalBlock_Basisgen_Std_Weak_M=N+1.txt};
			\addplot+[line width=1.5pt,dashed] table[y=std] 
				{./ThermalBlock_Basisgen_Std_Weak_M=N+1.txt};
			\addplot+[line width=1.5pt,dotted] table[y=hier] 
				{./ThermalBlock_Basisgen_Std_Weak_M=N+1.txt};
		\end{axis}
		\end{tikzpicture}
		\label{fig:Thermal_M2_weak_N+1}
	}
	\scriptsize{
	\begin{tabular}[b]{cc}
	\toprule
	$N$ & $\Theta_{N,N+1}$ \\ 
	\midrule
	$1$ & $0.9905$ \\ 
	$2$ & $0.9532$ \\ 
	$3$ & $0.9430$ \\ 
	$4$ & $0.7848$ \\ 
	$5$ & $0.7358$ \\ 
	$6$ & $0.6466$ \\ 
	$7$ & $0.7087$ \\ 
	$8$ & $0.7163$ \\ 
	$9$ & $0.9180$ \\ 
	$10$ & $0.4665$ \\ 
	\bottomrule
	\end{tabular} 
	}
	\subfigure[Weak greedy by $\Delta_N^{\text{Std}}$, $M = N + 2$.]{
		\begin{tikzpicture}[scale=0.45]
		\begin{axis}[
			width=0.77\textwidth,
			xlabel={$N$},
			ymode=log,
			cycle list name=color list,
			unbounded coords=jump,
			legend pos=outer north east,
			legend cell align=left,
			legend style={draw=none,font=\tiny},
			clip=false,xtick=data,
			y=0.435cm, 
			]
			\addplot+[line width=1.5pt] table[y=err] 
				{./ThermalBlock_Basisgen_Std_Weak_M=N+2.txt};
			\addplot+[line width=1.5pt,dashed] table[y=std] 
				{./ThermalBlock_Basisgen_Std_Weak_M=N+2.txt};
			\addplot+[line width=1.5pt,dotted] table[y=hier] 
				{./ThermalBlock_Basisgen_Std_Weak_M=N+2.txt};
		\end{axis}
		\end{tikzpicture}
		\label{fig:Thermal_M2_weak_N+2}
		}
	\scriptsize{
	\begin{tabular}[b]{cc}
	\toprule
	$N$ & $\Theta_{N,N+2}$ \\ 
	\midrule
	$1$ & $0.5350$ \\ 
	$2$ & $0.4114$ \\ 
	$3$ & $0.5183$ \\ 
	$4$ & $0.4044$ \\ 
	$5$ & $0.0469$ \\ 
	$6$ & $0.4526$ \\ 
	$7$ & $0.5077$ \\ 
	$8$ & $0.4279$ \\ 
	$9$ & $0.1000$ \\ 
	$10$ & $0.0193$ \\ 
	\bottomrule
	\end{tabular} 
	}
\end{figure}

\textbf{Online effectivity.} As we have seen that both $\Delta_N^\mathrm{Std}$ and $\Delta_{N,M}^\mathrm{Hier}$ (for appropriate values of $M$) are sharp, we investigate the online CPU time required to compute these error estimators. In order to do so, we consider the obtained effectivity $\eta$, i.e., the ratio of error estimator and true error for 100 test parameters. The results are shown in Figure \ref{fig:thermalblock_scatter}, where the values of $\eta$ are plotted over the required online time. The circles correspond to $\Delta_{N,M}^\mathrm{Hier}$ for different values of $M$. The few circles with $\eta>5$ correspond to quite small values of $M$ and large parameter sets. All remaining values cluster for effectivites below 2 and online CPU times of less than 0.1 seconds. As we can also see, the online CPU time is more or less independent of the choice of $M$. 
This is compared to $\Delta_N^\mathrm{Std}$. The online timings include also the SCM in this case.  The crosses in Figure \ref{fig:thermalblock_scatter} confirm the sharpness of the standard error estimator, but at the expense of CPU times which are about 15 times larger than for the hierarchical case.
\begin{figure}[!htb]
	\centering
		\begin{tikzpicture}[scale=0.8]
		\begin{axis}[
			width=0.99\textwidth,
			xlabel={time [s]},
			ylabel=$\eta$,
			y=0.12cm,
			]
			\addplot[only marks, mark=x,black, 
				mark options={mark size=3}] table[x=time,y=eta] 
				{./Basisgen_Strong_scm_1000_random_sampling_points.txt};
			\addplot[only marks, mark=o,black, 
				mark options={mark size=3}] table[x=time,y=eta] 
				{./Basisgen_Strong_N=1-9_M=N+3_1000_random_sampling_points.txt};
			\addplot[only marks, mark=o,black, 
				mark options={mark size=3}] table[x=time,y=eta] 
				{./Basisgen_Strong_N=1-10_M=N+2_1000_random_sampling_points.txt};
			\addplot[only marks, mark=o,black, 
			mark options={mark size=3}] table[x=time,y=eta] 
			{./Basisgen_Strong_N=1-11_M=N+1_1000_random_sampling_points.txt};			
		\end{axis}
		\end{tikzpicture}
		\caption{\label{fig:thermalblock_scatter} Online effectivity index $\eta$ over online CPU-time for thermal block on $\mathcal{P}^{(2)}$, strong greedy. Circles: Hierarchical error estimator for $M=N+1$, $M=N+2$ and $M=N+3$; crosses: Standard error estimator.}
\end{figure}
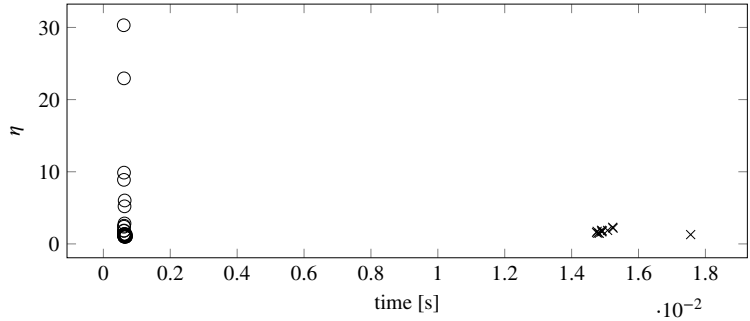



\subsection{Helmholtz Problem}
\label{sec:Helmholtz}
The Helmholtz equation arises from the time-dependent wave equation in the time-harmonic case, see e.g.\ \cite{Babuska1,Melenk,Babuska3,Babuska2} and references therein. Let $\Omega \subset \mathbb{R}^n$, $n \in \{1,2,3\}$, be a bounded Lipschitz domain with boundary $\Gamma := \partial \Omega$. For $\mu \in \mathcal{P} := [\mu_{\text{min}},\mu_{\text{max}}] \subset \mathbb{R}$ with $1 \leq \mu_{\text{min}} < \mu_{\text{max}} < \infty$, the Helmholtz problem reads
\begin{alignat}{3}
	\label{al:helmholtz}
	 -\Delta u(x) - \mu^2 u(x) 
	 	&= r(x), \qquad && x \in \Omega, \\
 	u (x) &= 0, \qquad && x \in \Gamma_D \subset \Gamma, \notag\\
 	\frac{\partial u}{\partial n}(x) + \mathbbit{i} \mu u(x) 
		&= g(x), \qquad && x \in \Gamma_R \subset \Gamma, \notag
\end{alignat}
where $\Gamma_D \cup \Gamma_R = \Gamma$. The parameter $\mu \in \mathcal{P}$ denotes the \emph{wavenumber}, defined by $\mu := \frac{\omega}{c}$ (SI unit: $m^{-1}$), where $\omega \in \mathbb{R}$ denotes the frequency and $c \in \mathbb{R}$ the wave propagation speed, $\mathbbit{i}:=\sqrt{-1}$. In high frequency problems the wavenumber is quite large resulting in   oscillations, see \cite{Melenk}. We use $\mu_{\text{max}} = 100$ here, since this suffices to show the desired effects. 
Test and trial spaces are again identical, $X=Y:= H_D^1(\Omega; \C) := \{ v \in H^1(\Omega; \C) : v \vert_{\Gamma_D} = 0\}$, but the sesquilinear form is no longer hermitean, i.e., 
\begin{align*}
	a ( u ,v;\mu) & = \int_{\Omega} \nabla u\cdot \overline{\nabla v} ~ dx 
		- \mu^2 \int_{\Omega} u \bar{v} ~ dx 
		+ \mathbbit{i} \mu \int_{\Gamma_R} u \bar{v} ~ ds, \\
	f ( v ; \mu ) & = \int_{\Omega} r \bar{v} ~ dx + \int_{\Gamma_R} g \bar{v} ~ dx.
\end{align*}
The affine decomposition in the form \eqref{eq:affinedecomp} is clear. Such problems are usually analyzed using the parameter-dependent norm given by
\begin{align*}
	\| v \|_{1,\mu}^2 := \mu^2 \| v \|^2_{0} + | v |_{1}^2,\qquad v\in H^1(\Omega;\C),
\end{align*}
which is equivalent to $\|\cdot\|_{1}$, i.e.,  
$\min\{ 1, \mu_{\min}\} \| v \|_{1}  \leq \| v \|_{1,\mu} \leq \max\{1, \mu_{\max}\} \| v \|_{1}$, $v \in H^1 (\Omega;\C)$, with coefficients, which depend on the parameter range, however. The well-posed\-ness is proven e.g.\ in \cite{Melenk} by the Fredholm alternative. Moreover, there exists a constant $C_{\text{$\inf$-$\sup$}}>0$ such that
\begin{align}\label{eq:Helmholtz-inf-sup}
	 \inf_{w \in X} \sup_{v \in Y} \frac{|a(w,v;\mu)|}{\| w \|_{1,\mu} \| v \|_{1,\mu}} 
	 &\geq \inf_{w \in X} \sup_{v \in Y} \frac{ \operatorname{Re}\{a(w,v;\mu)\}}{\| w \|_{1,\mu} \| v \|_{1,\mu}} \geq C_{\text{$\inf$-$\sup$}} \; \mu^{-\frac{7}{2}}.
\end{align}

For our numerical experiments we consider three cases of parameter spaces, namely 
$$
	\mathcal{P}^{(1)} = [1,5], \,\,
	\mathcal{P}^{(2)} = [95,100],\,\,
	\mathcal{P}^{(3)} = [90,100],
	\quad
	|\mathcal{P}_{\mathrm{train}}^{(i)}| = 10^4+1,\, i=1,2,3.
	$$ 
Thus, $\mathcal{P}^{(1)}$ is in the low-frequency domain so that the $\inf$-$\sup$ constant is expected to be moderate, whereas $\mathcal{P}^{(2)}$, $\mathcal{P}^{(3)}$ will lead to oscillatory, high-frequency solutions. The latter choices allow to investigate the dependency on the \emph{size} of the parameter set within the high-frequency regime. 
Our truth discretization is formed by spectral elements of degree 6 with 600 degrees of freedom for $\mathcal{P}^{(1)}$ (which turned out to be sufficient) and spectral elements of degree 16 with 16.000 degrees of freedom for $\mathcal{P}^{(2)}$ and $\mathcal{P}^{(3)}$. 

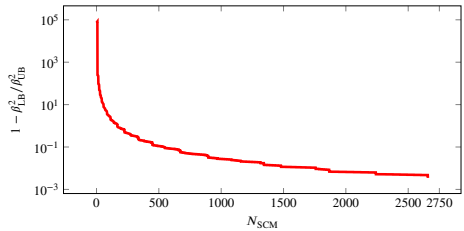
\begin{wrapfigure}[10]{r}{0.51\textwidth} 
		\begin{tikzpicture}[scale=0.49]
		\begin{axis}[
			width=0.95\textwidth,
			xlabel={$N_{\mathrm{SCM}}$},
			xtick={0,500,1000,1500,2000,2500,2750},
			ticklabel style={
    			/pgf/number format/.cd,
	    		use comma,
			1000 sep = {}
			  },
			ylabel={$1-{\beta_{\mathrm{LB}}^2}/{ \beta_{\mathrm{UB}}^2}$},
			ymode=log,
			cycle list name=color list,
			unbounded coords=jump,
			legend pos=outer north east,
			legend cell align=left,
			legend style={draw=none,font=\tiny},
			clip=false,
			y=0.25cm, 
			]
			\addplot+[line width=2pt] table[y=eta] 
				{./HeHo_95_100_SCM_conv_scm.txt};
		\end{axis}
		\end{tikzpicture}	
		\caption{\label{fig:helmholtz_scm_conv}SCM-convergence Helmholtz equation on $\mathcal{P}^{(2)}$.}
\end{wrapfigure} 
In order to compare the results concerning the hierarchical estimator with the \emph{best possible} standard one, we determined the involved discrete $\inf$-$\sup$ constant $\beta^{\mathcal{N}} (\mu)$ by computing the smallest eigenvalue of a generalized eigenvalue problem. As this is not online efficient, we used the SCM for the online comparisons 
in terms of CPU time. 
By \eqref{eq:Helmholtz-inf-sup}, we expect fairly small $\inf$-$\sup$ constants for large wavenumbers, which is expected to cause problems in $\Delta_N^\mathrm{Std}$. This fact is also mirrored by the poor convergence of the SCM shown in Figure \ref{fig:helmholtz_scm_conv}. For a good performance of $\Delta_N^\mathrm{Std}$ in terms of sharpness, one needs a good online approximation of $\beta^{\mathcal{N}} (\mu)$ resulting in many SCM iterations and large CPU times.

We start by describing the result for the low-frequency parameter set $\mathcal{P}^{(1)}$ and reduce ourselves to the strong greedy sampling since the results for the weak greedy with various error estimators turned out to be pretty much the same. As we can see in Figure \ref{fig:Helmholtz_M1_strong} both standard and hierarchical error estimator are quite sharp and the constants $\Theta_{N,M}$ are small -- overall a similar behavior as for the thermal block.
\begin{figure}[!htb]
	\caption{\label{fig:Helmholtz_M1_strong}Helmholtz equation, 
	$\mathcal{P}^{(1)} = [1,5]$, strong greedy.
		Average error over test set. 
		Red, solid: true error; blue, dashed: residual error estimator; 
		black, dotted: hierarchical error estimator for $M\in\{N+1, M = N+2\}$.}
	\subfigure[Strong greedy, $M = N + 1$.]{
		\begin{tikzpicture}[scale=0.45]
		\begin{axis}[
			width=0.75\textwidth,
			xlabel={$N$},
			ymode=log,
			cycle list name=color list,
			unbounded coords=jump,
			legend pos=outer north east,
			legend cell align=left,
			legend style={draw=none,font=\tiny},
			clip=false,xtick=data,
			y=0.382cm, 
			]
			\addplot+[line width=1.5pt] table[y=err] 
				{./HeHo_1_5_Basisgen_Strong_HeHo_Av_Err_data_M=N+1.txt};
			\addplot+[line width=1.5pt,dashed] table[y=std] 
				{./HeHo_1_5_Basisgen_Strong_HeHo_Av_Err_data_M=N+1.txt};
			\addplot+[line width=1.5pt,dotted] table[y=hier] 
				{./HeHo_1_5_Basisgen_Strong_HeHo_Av_Err_data_M=N+1.txt};
		\end{axis}
		\end{tikzpicture}
		\label{fig:Helmholtz_M1_strong_N+1}
	}
	\scriptsize{
	\begin{tabular}[b]{cc}
	\toprule
	$N$ & $\Theta_{N,N+1}$ \\ 
	\midrule
	$1$ & $0.3254$ \\ 
	$2$ & $0.6396$ \\ 
	$3$ & $0.1027$ \\ 
	$4$ & $0.3714$ \\ 
	$5$ & $0.0737$ \\ 
	\bottomrule
	\end{tabular} 
	}
	\subfigure[Strong greedy, $M = N + 2$.]{
		\begin{tikzpicture}[scale=0.45]
		\begin{axis}[
			width=0.75\textwidth,
			xlabel={$N$},
			ymode=log,
			cycle list name=color list,
			unbounded coords=jump,
			legend pos=outer north east,
			legend cell align=left,
			legend style={draw=none,font=\tiny},
			clip=false,xtick=data,
			y=0.41cm, 
			]
			\addplot+[line width=1.5pt] table[y=err] 
				{./HeHo_1_5_Basisgen_Strong_HeHo_Av_Err_data_M=N+2.txt};
			\addplot+[line width=1.5pt,dashed] table[y=std] 
				{./HeHo_1_5_Basisgen_Strong_HeHo_Av_Err_data_M=N+2.txt};
			\addplot+[line width=1.5pt,dotted] table[y=hier] 
				{./HeHo_1_5_Basisgen_Strong_HeHo_Av_Err_data_M=N+2.txt};
		\end{axis}
		\end{tikzpicture}
		\label{fig:Helmholtz_M1_strong_N+2}
		}
	\scriptsize{
	\begin{tabular}[b]{cc}
	\toprule
	$N$ & $\Theta_{N,N+2}$ \\ 
	\midrule
	$1$ & $0.1211$ \\ 
	$2$ & $0.0657$ \\ 
	$3$ & $0.0178$ \\ 
	$4$ & $0.0103$ \\ 
	$5$ & $0.0228$ \\ 
	\bottomrule
	\end{tabular} 
	}
\end{figure}
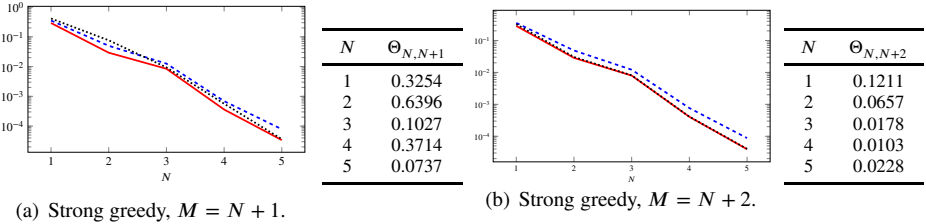 

Next, we consider the (smaller) high frequency parameter set $\mathcal{P}^{(2)}$ and again restrict ourselves to the strong greedy sampling (the results for different versions of the weak are again quite similar). First, we note that the minimal choice of $M=N+1$ for the hierarchical error estimator is not sufficient in order to yield sharp estimates as can be seen in the left graph in Figure \ref{fig:Helmholtz_M2_strong}. We have also found that the saturation property cannot be guaranteed numerically in this case. In the right graph, we thus use a Lagrange basis with $M=N+2$ and obtain bounds that are even better than for the standard estimator. Recall, that the blue dashed line for $\Delta_N^{\text{Std}}$ is w.r.t.\ to a high-fidelity approximation for the $\inf$-$\sup$ constant, i.e., the best possible standard residual-based error bound. Also the values for $\Theta_{N,M}$ are quite good. Thus, $\Delta_{N,N+2}^\mathrm{Hier}$ is a cheap and sharp error bound even for the high-frequency case.
\begin{figure}[!htb]
	\caption{\label{fig:Helmholtz_M2_strong}Helmholtz equation, 
		$\mathcal{P}^{(2)} = [95,100]$, 
		strong greedy.
		Average error over test set of parameters. 
		Red, solid: true error; 
		blue, dashed: residual error estimator; 
		black, dotted: hierarchical error estimator with $M\in\{N+1, N+2\}$.}
	\subfigure[Strong greedy, $M = N + 1$.]{
		\begin{tikzpicture}[scale=0.45]
		\begin{axis}[
			width=0.75\textwidth,
			xlabel={$N$},
			ymode=log,
			cycle list name=color list,
			unbounded coords=jump,
			legend pos=outer north east,
			legend cell align=left,
			legend style={draw=none,font=\tiny},
			clip=false,xtick=data,
			y=0.495cm, 
			]
			\addplot+[line width=1.5pt] table[y=err] 
			{./HeHo_95_100_Basisgen_Strong_HeHo_Av_Err_data_M=N+1.txt};
			\addplot+[line width=1.5pt,dashed] table[y=std] 
			{./HeHo_95_100_Basisgen_Strong_HeHo_Av_Err_data_M=N+1.txt};
			\addplot+[line width=1.5pt,dotted] table[y=hier] 
			{./HeHo_95_100_Basisgen_Strong_HeHo_Av_Err_data_M=N+1.txt};
		\end{axis}
		\end{tikzpicture}
		\label{fig:Helmholtz_M2_strong_N+1}
	}
	\scriptsize{
	\begin{tabular}[b]{cc}
	\toprule
	$N$ & $\Theta_{N,N+1}$ \\ 
	\midrule
	$1$ & $0.7075$ \\ 
	$2$ & $0.8407$ \\ 
	$3$ & $0.7173$ \\ 
	$4$ & $0.9563$ \\ 
	$5$ & $0.2672$ \\ 
	$6$ & $0.2088$ \\ 
	\bottomrule
	\end{tabular} 
	}
	\subfigure[Strong greedy, $M = N + 2$.]{
		\begin{tikzpicture}[scale=0.45]
		\begin{axis}[
			width=0.75\textwidth,
			xlabel={$N$},
			ymode=log,
			cycle list name=color list,
			unbounded coords=jump,
			legend pos=outer north east,
			legend cell align=left,
			legend style={draw=none,font=\tiny},
			clip=false,xtick=data,
			y=0.535cm, 
			]
			\addplot+[line width=1.5pt] table[y=err] 
			{./HeHo_95_100_Basisgen_Strong_HeHo_Av_Err_data_M=N+2.txt};
			\addplot+[line width=1.5pt,dashed] table[y=std] 
			{./HeHo_95_100_Basisgen_Strong_HeHo_Av_Err_data_M=N+2.txt};
			\addplot+[line width=1.5pt,dotted] table[y=hier] 
			{./HeHo_95_100_Basisgen_Strong_HeHo_Av_Err_data_M=N+2.txt};
		\end{axis}
		\end{tikzpicture}
		\label{fig:Helmholtz_M2_strong_N+2}
		}
	\scriptsize{
	\begin{tabular}[b]{cc}	
	\toprule
	$N$ & $\Theta_{N,N+2}$ \\ 
	\midrule
	$1$ & $0.5948$ \\ 
	$2$ & $0.3390$ \\ 
	$3$ & $0.3770$ \\ 
	$4$ & $0.0610$ \\ 
	$5$ & $0.0120$ \\ 
	$6$ & $0.0093$ \\ 
	\bottomrule
	\end{tabular}
	}
\end{figure}
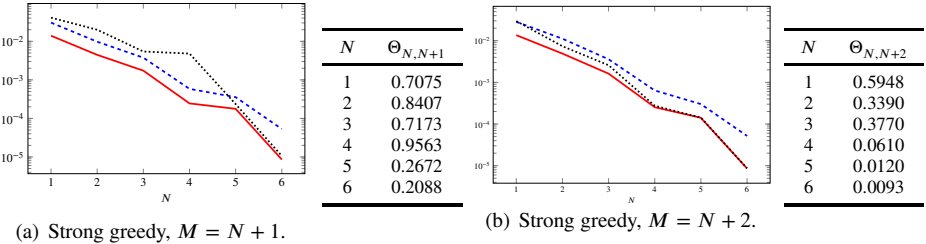

Finally, we consider $\mathcal{P}^{(3)}$, which is a high frequency parameter set of doubled size as compared to $\mathcal{P}^{(2)}$. The error plots for the strong greedy sampling are shown in Figure \ref{fig:Helmholtz_M3_strong}. In this case, the Lagrange-based space $X_M$ for $M=N+2$ only yields reasonable results for $N\geq4$ (for smaller values, the saturation is not guaranteed), but then $\Delta_{N,N+2}^\mathrm{Hier}$ outperforms $\Delta_N^{\mathrm{Std}}$ in terms of accuracy. As we can see from the right-hand side of the figure, $M=N+3$ gives quite sharp results for $N\geq 3$. Again, for smaller values of $N$, the saturation is not justified.
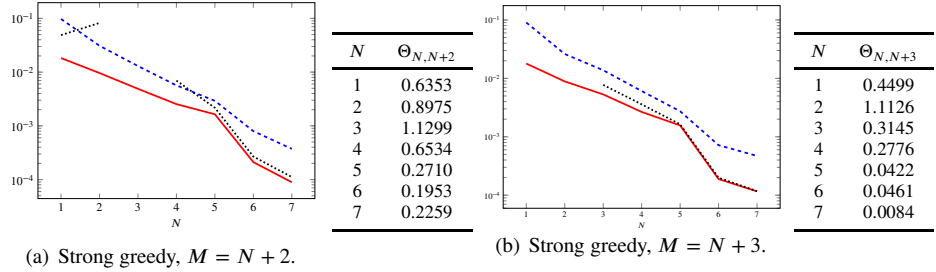
\begin{figure}[!htb]
	\caption{\label{fig:Helmholtz_M3_strong}Helmholtz equation, 
		$\mathcal{P}^{(3)} = [90,100]$, 
		strong greedy.
		Average error over test set of parameters. 
		Red, solid: true error; 
		blue, dashed: residual error estimator; 
		black, dotted: hierarchical error estimator, $M\in\{ N+2, N+3\}$.}
	\subfigure[Strong greedy, $M = N + 2$.]{
		\begin{tikzpicture}[scale=0.45]
		\begin{axis}[
			width=0.75\textwidth,
			xlabel={$N$},
			ymode=log,
			cycle list name=color list,
			unbounded coords=jump,
			legend pos=outer north east,
			legend cell align=left,
			legend style={draw=none,font=\tiny},
			clip=false,xtick=data,
			y=0.69cm, 
			]
			\addplot+[line width=1.5pt] table[y=err] 
				{./HeHo_90_100_Basisgen_Strong_heho_90_100_Strong_M=N+2.txt};
			\addplot+[line width=1.5pt,dashed] table[y=std] 
				{./HeHo_90_100_Basisgen_Strong_heho_90_100_Strong_M=N+2.txt};
			\addplot+[line width=1.5pt,dotted] table[y=hier] 
				{./HeHo_90_100_Basisgen_Strong_heho_90_100_Strong_M=N+2.txt};
		\end{axis}
		\end{tikzpicture}
		\label{fig:Helmholtz_M3_strong_N+2}
	}
	\scriptsize{
	\begin{tabular}[b]{cc}
	\toprule
	$N$ & $\Theta_{N,N+2}$ \\ 
	\midrule
	$1$ & $0.6353$ \\ 
	$2$ & $0.8975$ \\ 
	$3$ & $1.1299$ \\ 
	$4$ & $0.6534$ \\ 
	$5$ & $0.2710$ \\ 
	$6$ & $0.1953$ \\ 
	$7$ & $0.2259$ \\ 
	\bottomrule
	\end{tabular} 
	}
	\subfigure[Strong greedy, $M = N + 3$.]{
		\begin{tikzpicture}[scale=0.45]
		\begin{axis}[
			width=0.75\textwidth,
			xlabel={$N$},
			ymode=log,
			cycle list name=color list,
			unbounded coords=jump,
			legend pos=outer north east,
			legend cell align=left,
			legend style={draw=none,font=\tiny},
			clip=false,xtick=data,
			y=0.75cm, 
			]
			\addplot+[line width=1.5pt] table[y=err] 
				{./HeHo_90_100_Basisgen_Strong_heho_90_100_Strong_M=N+3.txt};
			\addplot+[line width=1.5pt,dashed] table[y=std] 
				{./HeHo_90_100_Basisgen_Strong_heho_90_100_Strong_M=N+3.txt};
			\addplot+[line width=1.5pt,dotted] table[y=hier] 
				{./HeHo_90_100_Basisgen_Strong_heho_90_100_Strong_M=N+3.txt};
		\end{axis}
		\end{tikzpicture}
		\label{fig:Helmholtz_M3_strong_N+3}
		}
	\scriptsize{
	\begin{tabular}[b]{cc}
	\toprule
	$N$ & $\Theta_{N,N+3}$ \\ 
	\midrule
	$1$ & $0.4499$ \\ 
	$2$ & $1.1126$ \\ 
	$3$ & $0.3145$ \\ 
	$4$ & $0.2776$ \\ 
	$5$ & $0.0422$ \\ 
	$6$ & $0.0461$ \\ 
	$7$ & $0.0084$ \\ 
	\bottomrule
	\end{tabular} 
	}
\end{figure}

Due to the lack of saturation for the Lagrange-type construction, we also tested the Taylor approach. We obtained even better results for all parameter sets. For $\mathcal{P}^{(3)}$, we display the results of a weak greedy sampling in Figure \ref{fig:Helmholtz_M3_HierErrEst_Derivatives}. Even for $K_n = 2$, we got good results as can be seen by the fact that the values of $\Theta_{N,M}$ are close to zero. Moreover, $\Delta_{N,M}^\mathrm{Hier}$ is quite sharp. The situation even improves for $K_n=3$ in terms of sharpness for small $N$.
\begin{figure}[!htb]
	\caption{\label{fig:Helmholtz_M3_HierErrEst_Derivatives}
	Helmholtz equation, $\mathcal{P}^{(3)} = [90,100]$, 
	weak greedy with parameter sampling via hierarchical error estimator.
		Average error over test set of parameters. 
		Red, solid: true error; blue, dashed: residual error estimator; 
		black, dotted: hierarchical error estimator, with Taylor basis, $K_n \in \{2,3\}$.}
	\subfigure[Weak greedy by $\frac{\Delta_{N,M}(\mu)}{1 - \Theta_{N,M}^{\mathcal{N}}}$, 
		 $K_n = 2$.]{
		\begin{tikzpicture}[scale=0.45]
		\begin{axis}[
			width=0.75\textwidth,
			xlabel={$N$},
			ymode=log,
			cycle list name=color list,
			unbounded coords=jump,
			legend pos=outer north east,
			legend cell align=left,
			legend style={draw=none,font=\tiny},
			clip=false,xtick=data,
			y=0.553cm, 
			]
			\addplot+[line width=1.5pt] table[y=err] 
				{./HeHo_90_100_Basisgen_Hier_Helmholtz_90_100_2Ableitungen.txt};
			\addplot+[line width=1.5pt,dashed] table[y=std] 
				{./HeHo_90_100_Basisgen_Hier_Helmholtz_90_100_2Ableitungen.txt};
			\addplot+[line width=1.5pt,dotted] table[y=hier] 
				{./HeHo_90_100_Basisgen_Hier_Helmholtz_90_100_2Ableitungen.txt};
		\end{axis}
		\end{tikzpicture}
		\label{fig:Helmholtz_M3_HierErrEst_twoDerivatives}
	}
	\scriptsize{
	\begin{tabular}[b]{cc}
	\toprule
	$N$ & $\Theta_{N,M}$ \\ 
	\midrule
	$1$ & $0.8435$ \\ 
	$2$ & $0.9789$ \\ 
	$3$ & $0.4281$ \\ 
	$4$ & $0.1216$ \\ 
	$5$ & $0.0003$ \\ 
	$6$ & $0.0014$ \\ 
	$7$ & $0.0017$ \\ 
	$8$ & $0.0247$ \\ 
	$9$ & $0.0131$ \\ 
	$10$ & $0.0834$ \\ 
	$11$ & $0.0391$ \\ 
	$12$ & $0.0399$ \\ 
	\bottomrule
	\end{tabular} 
	}
	\subfigure[Weak greedy by $\frac{\Delta_{N,M}(\mu)}{1 - \Theta_{N,M}^{\mathcal{N}}}$,  
		 $K_n = 3$.]{
		\begin{tikzpicture}[scale=0.45]
		\begin{axis}[
			width=0.77\textwidth,
			xlabel={$N$},
			ymode=log,
			cycle list name=color list,
			unbounded coords=jump,
			legend pos=outer north east,
			legend cell align=left,
			legend style={draw=none,font=\tiny},
			clip=false,xtick=data,
			y=0.512cm, 
			]
			\addplot+[line width=1.5pt] table[y=err] 
				{./HeHo_90_100_Basisgen_Hier_Helmholtz_90_100_3Ableitungen.txt};
			\addplot+[line width=1.5pt,dashed] table[y=std] 
				{./HeHo_90_100_Basisgen_Hier_Helmholtz_90_100_3Ableitungen.txt};
			\addplot+[line width=1.5pt,dotted] table[y=hier] 
				{./HeHo_90_100_Basisgen_Hier_Helmholtz_90_100_3Ableitungen.txt};
		\end{axis}
		\end{tikzpicture}
		\label{fig:Helmholtz_M3_HierErrEst_threeDerivatives} 
	}
	\scriptsize{
	\begin{tabular}[b]{cc}
	\toprule
	$N$ & $\Theta_{N,M}$ \\ 
	\midrule
	$1$ & $0.8196$ \\ 
	$2$ & $0.7687$ \\ 
	$3$ & $0.1058$ \\ 
	$4$ & $0.0000$ \\ 
	$5$ & $0.0001$ \\ 
	$6$ & $0.0029$ \\ 
	$7$ & $0.0052$ \\ 
	$8$ & $0.0160$ \\ 
	$9$ & $0.0350$ \\
	$10$ & $0.0401$ \\ 
	$11$ & $0.0221$ \\ 
	$12$ & $0.0418$ \\  
	\bottomrule
	\end{tabular} 
	}
\end{figure}
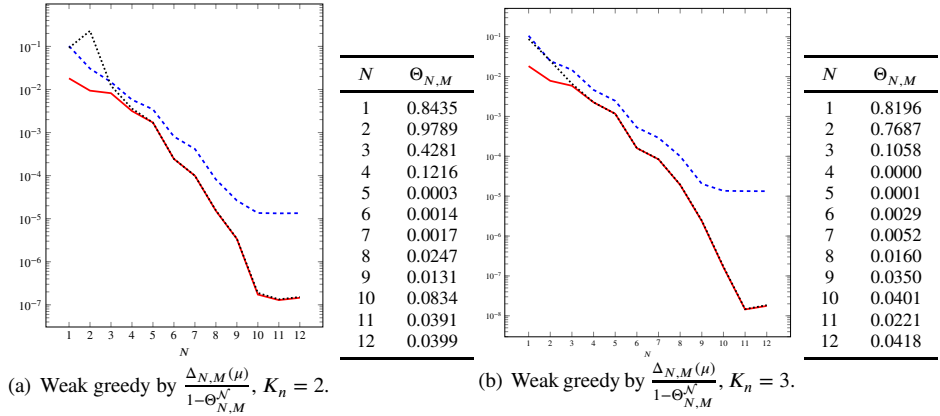

\textbf{Online effectivity.} 
As before in \S \ref{sec:ThermalBlock} for the thermal block, we compare the online efficiencies of standard and hierarchical error estimator, see Figure \ref{fig:Helmholtz_scatter}. First, note that we could not include values for the larger high-frequency parameter range $\mathcal{P}^{(3)}$ there, since the SCM required for $\Delta_N^{\text{Std}}$ did not converge, which means that the standard bound cannot be used in an online-efficient manner.\footnote{In addition, the SCM did not converge at all using a discontinuous Galerkin truth discretization.}

In Figure \ref{fig:Helmholtz_scatter}, we show the effectivity over the online CPU time, again for $\Delta_N^{\text{Std}}$ by crosses and for $\Delta_{N,M}^\mathrm{Hier}$ (for different values of $M$) by circles. First, we note that the values of $M$ almost do not influence the CPU times, so that we can easily adjust the accuracy, as before. Moreover, the accuracies of both bounds are quite comparable, but the computation of $\Delta_{N,M}^\mathrm{Hier}$ is much faster.
\begin{figure}[!htb]
	\centering
		\begin{tikzpicture}[scale=0.8]
		\begin{axis}[
			width=0.99\textwidth,
			xlabel={time [s]},
			ylabel=$\eta$,
			y=0.05cm,
			xtick={0,0.02,0.04,0.06,0.08,0.1},
			]
			\addplot[only marks, mark=x,black, mark options={mark size=3}] table[x=time,y=scmTOL] {./HeHo_95_100_ScatterPlot_Basisgen_Strong_N=4_mean_time_over_mean_efficiency_std_err_est_1000_random_sampling_points.txt};
			\addplot[only marks,mark=o,black, mark options={mark size=3}] table[x=time,y=eta] {./HeHo_95_100_ScatterPlot_Basisgen_Strong_N=4_M=N+1_M=N+2_M=N+3_mean_time_over_mean_efficiency_hier_err_est_1000_random_sampling_points.txt};
			\addplot[only marks,mark=x,black, mark options={mark size=3}] table[x=time,y=scmTOL] {./HeHo_95_100_ScatterPlot_Basisgen_Strong_N=5_mean_time_over_mean_efficiency_std_err_est_1000_random_sampling_points.txt};
			\addplot[only marks,mark=o,black, mark options={mark size=3}] table[x=time,y=eta] {./HeHo_95_100_ScatterPlot_Basisgen_Strong_N=5_M=N+1_M=N+2_M=N+3_mean_time_over_mean_efficiency_hier_err_est_1000_random_sampling_points.txt};
			\addplot[only marks,mark=x,black, mark options={mark size=3}] table[x=time,y=scmTOL] {./HeHo_95_100_ScatterPlot_Basisgen_Strong_N=6_mean_time_over_mean_efficiency_std_err_est_1000_random_sampling_points.txt};
			\addplot[only marks,mark=o,black, mark options={mark size=3}] table[x=time,y=eta] {./HeHo_95_100_ScatterPlot_Basisgen_Strong_N=6_M=N+1_M=N+2_mean_time_over_mean_efficiency_hier_err_est_1000_random_sampling_points.txt};		
			\addplot[only marks, mark=x,black, mark options={mark size=3}] table[x=time,y=scmTOL] {./HeHo_1_5_ScatterPlot_Basisgen_Strong_N=4_mean_time_over_mean_efficiency_std_err_est_1000_random_sampling_points.txt};
			\addplot[only marks,mark=o,black, mark options={mark size=3}] table[x=time,y=M] {./HeHo_1_5_ScatterPlot_Basisgen_Strong_N=4_M=N+1_M=N+2_mean_time_over_mean_efficiency_hier_err_est_1000_random_sampling_points.txt};
			\addplot[only marks,mark=x,black, mark options={mark size=3}] table[x=time,y=scmTOL] {./HeHo_1_5_ScatterPlot_Basisgen_Strong_N=3_mean_time_over_mean_efficiency_std_err_est_1000_random_sampling_points.txt};
			\addplot[only marks,mark=o,black, mark options={mark size=3}] table[x=time,y=M] {./HeHo_1_5_ScatterPlot_Basisgen_Strong_N=3_M=N+1_M=N+2_M=N+3_mean_time_over_mean_efficiency_hier_err_est_1000_random_sampling_points.txt};			
		\end{axis}
		\end{tikzpicture}
		\caption{\label{fig:Helmholtz_scatter} Effectivity index $\eta$ over online CPU-time 
			for Helmholtz problem on $\mathcal{P}^{(1)}$, $\mathcal{P}^{(2)}$;
		 strong greedy sampling. 
		 Circles: Hierarchical error estimator for different $M$; 
		 crosses: Standard error estimator.}
\end{figure}
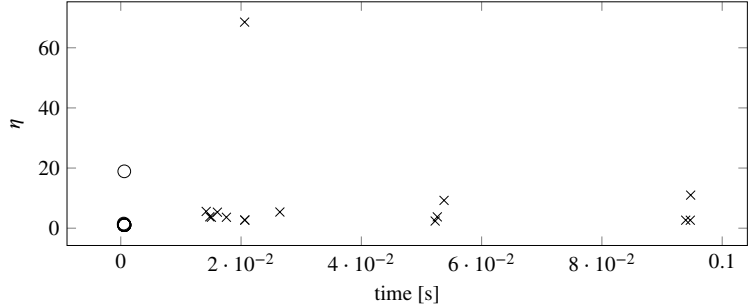

\subsection{Conclusions}
Let us come back to the questions from the beginning of this section: 
\begin{enumerate}
	\item \emph{How is the performance of $\Delta_{N,M}^\mathrm{Hier}$ as compared to $\Delta_N^\mathrm{Std}$?}\newline
	Even for those cases that are in favor of $\Delta_N^\mathrm{Std}$ (stable with precise knowledge of the $\inf$-$\sup$ constant), $\Delta_{N,M}^\mathrm{Hier}$ turned out to yield a sharp error bound and to be online efficient. The potential becomes even more pronounced for problems with bad $\inf$-$\sup$ behavior.
	\item \emph{How does the performance depend on a sharp lower $\inf$-$\sup$ bound?}\newline
	The poorer the $\inf$-$\sup$ estimate is, the more $\Delta_N^\mathrm{Std}$ is outperformed by $\Delta_{N,M}^\mathrm{Hier}$ -- in terms of sharpness and efficiency.
	\item \emph{What is a reasonable choice for the constant $\Theta_{N,M}$?}\newline
	In all tested examples, we got very reasonable values for $\Theta_{N,M}$, provided that the saturation holds. However, even the determination via a test set requires the computation of possibly many truth solutions, the optimization problem \eqref{eq:rootfindingproblem} for the verification of the saturation and the computation of $\Theta_{M,N}^{\mathcal{N}}$ is quite costly, even though done offline. But our results show that it might be sufficient to do this on a fairly small test set since we got nice results in all case.
	\item \emph{What is a good choice for $X_M$?}\newline
	In all investigated cases, $M$ could be chosen quite moderate. This is due to the fact that our problems are of elliptic flavor even in the Helmholtz case. In \cite{2017arXiv170404139F,RB:Transport} for problems involving transport phenomena, $X_M$ has to be chosen significantly larger. However, we have also seen that even for problems with very small $\inf$-$\sup$ constant, $X_M$ can be chosen reasonably small. Moreover, the online CPU-times seem almost independent on the choice of $X_M$ and are much smaller as for computing $\Delta_N^\mathrm{Std}$ using the SCM (if the SCM converges at all).
	
	We compared also Lagrange- and Taylor-type approaches to construct $X_M$. Trying to use the Lagrange approach within parameter sampling using a weak greedy approach resulted in multiple selections of snapshots and non-guaranteed saturation. Both problems could be resolved using the Taylor approach, which, however, requires a certain regularity of $u$ with respect to the parameter. In this case, for a fixed $N$, we are able to improve the effectivity by increasing the order of derivatives. 
\end{enumerate}

\nocite{*}
\bibliographystyle{abbrv}
\bibliography{HORU_references.bib}

\begin{thebibliography}{10}

\bibitem{WaveletRB}
M.~Ali, K.~Steih, and K.~Urban.
\newblock Reduced basis methods with adaptive snapshot computations.
\newblock {\em Adv.\ Comp.\ Math.}, pages 1--38, 2016.

\bibitem{Babuska1}
I.~M. Babu{\v{s}}ka and S.~A. Sauter.
\newblock {Is the Pollution Effect of the FEM Avoidable for the Helmholtz
  Equation Considering High Wave Numbers?}
\newblock {\em SIAM J.\ Numer.\ Anal.}, 34(6):2392--2423, 1997.

\bibitem{BankHier1}
R.~E. Bank and R.~K. Smith.
\newblock A posteriori error estimates based on hierarchical bases.
\newblock {\em SIAM J.\ Numer.\ Anal.}, 30(4):921--935, 1993.

\bibitem{EIM}
M.~Barrault, Y.~Maday, N.~C. Nguyen, and A.~T. Patera.
\newblock {An 'Empirical interpolation' method: Application to efficient
  reduced-basis discretization of partial differential equations}.
\newblock {\em C.R.\ Acad.\ Sci.\ Math.}, 339(9):667 -- 672, 2004.

\bibitem{ConvergenceRB}
P.~Binev, A.~Cohen, W.~Dahmen, R.~DeVore, G.~Petrova, and P.~Wojtaszczyk.
\newblock Convergence rates for greedy algorithms in reduced basis methods.
\newblock {\em SIAM J.\ Math.\ Anal.}, 43(3):1457--1472, 2011.

\bibitem{RB:Transport}
J.~Brunken, K.~Smetana, and K.~Urban.
\newblock Parameterized first order transport equations: Realization of
  optimally stable petrov-galerkin methods.
\newblock Ulm Univ., preprint, 2017.

\bibitem{MR2877366}
A.~Buffa, Y.~Maday, A.~T. Patera, C.~Prud'homme, and G.~Turinici.
\newblock {\it {A} priori} convergence of the greedy algorithm for the
  parametrized reduced basis method.
\newblock {\em ESAIM Math.\ Model.\ Numer.\ Anal.}, 46(3):595--603, 2012.

\bibitem{Buhr20144094}
A.~Buhr, C.~Engwer, M.~Ohlberger, and S.~Rave.
\newblock A numerically stable a posteriori error estimator for reduced basis
  approximations of elliptic equations.
\newblock {\em 11th World Congress on Computational Mechanics, WCCM 2014, 5th
  European Conference on Computational Mechanics, ECCM 2014 and 6th European
  Conference on Computational Fluid Dynamics, ECFD 2014}, pages 4094--4102,
  2014.

\bibitem{CanutoSM1}
C.~Canuto.
\newblock {\em Spectral methods: fundamentals in single domains}.
\newblock Springer, Berlin; Heidelberg; New York, 2006.

\bibitem{CanutoSM2}
C.~Canuto.
\newblock {\em Spectral methods: evolution to complex geometries and
  applications to fluid dynamics}.
\newblock Springer, Berlin; Heidelberg, 2007.

\bibitem{CTU09}
C.~Canuto, T.~Tonn, and K.~Urban.
\newblock A posteriori error analysis of the reduced basis method for nonaffine
  parametrized nonlinear {PDE}s.
\newblock {\em SIAM J. Numer. Anal.}, 47(3):2001--2022, 2009.

\bibitem{CHEN20081295}
Y.~Chen, J.~S. Hesthaven, Y.~Maday, and J.~Rodr{\'i}guez.
\newblock A monotonic evaluation of lower bounds for inf-sup stability
  constants in the frame of reduced basis approximations.
\newblock {\em C.R.\ Acad.\ Sci.\ Math.}, 346(23):1295 -- 1300, 2008.

\bibitem{MaxwellSCM}
Y.~Chen, J.~S. Hesthaven, Y.~Maday, and J.~Rodr{\'i}guez.
\newblock {Improved successive constraint method based a posteriori error
  estimate for reduced basis approximation of 2D Maxwell's problem}.
\newblock {\em ESAIM Math.\ Model.\ Numer.\ Anal.}, 43(6):1099--1116, 2009.

\bibitem{MR1392158}
J.~R. Cho and J.~T. Oden.
\newblock A priori modeling error estimates of hierarchical models for
  elasticity problems for plate- and shell-like structures.
\newblock {\em Math. Comput. Modelling}, 23(10):117--133, 1996.

\bibitem{Dinkelbach}
W.~Dinkelbach.
\newblock On nonlinear fractional programming.
\newblock {\em Management Science}, 13(7):492--498, 1967.

\bibitem{MR2950678}
C.~Dom{\'\i}nguez, E.~P. Stephan, and M.~Maischak.
\newblock A {FE}-{BE} coupling for a fluid-struct{\-}ure interaction problem:
  hierarchical a posteriori error estimates.
\newblock {\em Numer. Methods Partial Differential Equations},
  28(5):1417--1439, 2012.

\bibitem{DHO12}
M.~Drohmann, B.~Haasdonk, and M.~Ohlberger.
\newblock Reduced basis approximation for nonlinear parametrized evolution
  equations based on empirical operator interpolation.
\newblock {\em SIAM J. Sci. Comput.}, 34(2):A937--A969, 2012.

\bibitem{HPParabolicPDE}
J.~L. Eftang, D.~J. Knezevic, and A.~T. Patera.
\newblock An hp certified reduced basis method for parametrized parabolic
  partial differential equations.
\newblock {\em Mathematical and Computer Modelling of Dynamical Systems},
  17(4):395--422, 2011.

\bibitem{MR2746617}
J.~L. Eftang, A.~T. Patera, and E.~M. R{\o}nquist.
\newblock An `{$hp$}' certified reduced basis method for parametrized elliptic
  partial differential equations.
\newblock {\em In: SIAM J.\ Sci.\ Comput.}, 32(6):3170--3200, 2010.

\bibitem{Melenk}
S.~Esterhazy and J.~M. Melenk.
\newblock {\em On Stability of Discretizations of the Helmholtz Equation},
  pages 285--324.
\newblock Springer Berlin Heidelberg, Berlin, Heidelberg, 2012.

\bibitem{2017arXiv170404139F}
J.~{Feinauer}, S.~{Hein}, S.~{Rave}, S.~{Schmidt}, D.~{Westhoff}, J.~{Zausch},
  O.~{Iliev}, A.~{Latz}, M.~{Ohlberger}, and V.~{Schmidt}.
\newblock {MULTIBAT: Unified workflow for fast electrochemical 3D simulations
  of lithium-ion cells combining virtual stochastic microstructures,
  electrochemical degradation models and model order reduction}.
\newblock {\em ArXiv e-prints}, Apr. 2017.

\bibitem{RB:Wave}
S.~Glas, A.~Patera, and K.~Urban.
\newblock Reduced basis methods for the wave equation.
\newblock Unpublished manuscript, 2017.

\bibitem{Haasdonk:RB}
B.~Haasdonk.
\newblock {Reduced Basis Methods for Parametrized PDEs --- A Tutorial}.
\newblock In P.~Benner, A.~Cohen, M.~Ohlberger, and K.~Willcox, editors, {\em
  Model Reduction and Approximation}, chapter~2, pages 65--136. SIAM,
  Philadelphia, 2017.

\bibitem{HDO11}
B.~Haasdonk, M.~Dihlmann, and M.~Ohlberger.
\newblock A training set and multiple bases generation approach for
  parameterized model reduction based on adaptive grids in parameter space.
\newblock {\em Math. Comput. Model. Dyn. Syst.}, 17(4):423--442, 2011.

\bibitem{RozzaRB}
J.~S. Hesthaven, G.~Rozza, and B.~Stamm.
\newblock {\em Certified Reduced Basis Methods for Parame{\-}trized Partial
  Differential Equations}.
\newblock Springer International Publishing, 2016.

\bibitem{HesthavenEFI}
J.~S. Hesthaven, B.~Stamm, and S.~Zhang.
\newblock Certified {R}educed {B}asis {M}ethod for the {E}lectric {F}ield
  {I}ntegral {E}quation.
\newblock {\em SIAM J.\ Sci.\ Comput.}, 34(3):A1777--A1799, 2012.

\bibitem{HesthavenNDG}
J.~S. Hesthaven and T.~Warburton.
\newblock {\em Nodal Discontinuous Galerkin Methods: Algorithms, Analysis, and
  Applications}.
\newblock Springer Publishing Company, Incorporated, 1st edition, 2007.

\bibitem{Huang2011}
Y.~Huang, H.~Wei, W.~Yang, and N.~Yi.
\newblock {\em A New a Posteriori Error Estimate for Adaptive Finite Element
  Methods}, pages 63--74.
\newblock Springer Berlin Heidelberg, Berlin, Heidelberg, 2011.

\bibitem{SCM}
D.~B.~P. Huynh, G.~Rozza, S.~Sen, and A.~T. Patera.
\newblock A successive constraint linear optimization method for lower bounds
  of parametric coercivity and inf-sup stability constants.
\newblock {\em C.R.\ Acad.\ Sci.\ Math.}, 345(8):473 -- 478, 2007.

\bibitem{Babuska2}
F.~Ihlenburg and I.~Babu{\v{s}}ka.
\newblock {Finite element solution of the Helmholtz equation with high wave
  number Part I: The h-version of the FEM}.
\newblock {\em Comp.\ Math.\ Appl.}, 30(9):9 -- 37, 1995.

\bibitem{Babuska3}
F.~Ihlenburg and I.~Babu{\v{s}}ka.
\newblock {Finite Element Solution of the Helmholtz Equation with High Wave
  Number Part II: The h-p Version of the FEM}.
\newblock {\em SIAM J.\ Numer.\ Anal.}, 34(1):315--358, 1997.

\bibitem{OR16}
M.~Ohlberger and S.~Rave.
\newblock Reduced basis methods: Success, limitations and future challenges.
\newblock {\em Proceedings of the Conference Algoritmy}, pages 1--12, 2016.

\bibitem{Ohlberger:true}
M.~Ohlberger, S.~Rave, and F.~Schindler.
\newblock True error control for the localized reduced basis method for
  parabolic problems.
\newblock In {\em Model Reduction of Parametrized Systems}, pages 169--182.
  Springer International Publishing, Cham, 2017.

\bibitem{MR3431132}
M.~Ohlberger and F.~Schindler.
\newblock Error control for the localized reduced basis multiscale method with
  adaptive on-line enrichment.
\newblock {\em SIAM J. Sci. Comput.}, 37(6):A2865--A2895, 2015.

\bibitem{Patera}
A.~Patera and G.~Rozza.
\newblock {\em Reduced Basis Approximation and A Posteriori Error Estimation
  for Parametrized Partial Differential Equations}.
\newblock MIT, Cambridge (MA), USA, 2006.
\newblock Version 1.0.

\bibitem{MR611953}
P.~J. Prince and J.~R. Dormand.
\newblock High order embedded {R}unge-{K}utta formulae.
\newblock {\em J. Comput. Appl. Math.}, 7(1):67--75, 1981.

\bibitem{QuarteroniRB}
A.~Quarteroni, A.~Manzoni, and F.~Negri.
\newblock {\em Reduced basis methods for partial differential equations: An
  introduction}.
\newblock Springer International Publishing, Cham; Heidelberg, 2016.

\bibitem{DGBetrice}
B.~Rivi\`ere.
\newblock {\em Discontinuous Galerkin Methods for Solving Elliptic and
  Parabolic Equations}.
\newblock Society for Industrial and Applied Mathematics, 2008.

\bibitem{SpaceTimeSchwabStevenson}
C.~Schwab and R.~Stevenson.
\newblock Space-time adaptive wavelet methods for parabolic evolution problems.
\newblock {\em Mathematics of Computation}, 78(267):1293--1318, 2009.

\bibitem{SpaceTimeUrbanPatera}
K.~Urban and A.~T. Patera.
\newblock An improved error bound for reduced basis approximation of linear
  parabolic problems.
\newblock {\em Mathematics of Computation}, 83(288):1599--1615, 2014.

\bibitem{Oli}
K.~Urban, S.~Volkwein, and O.~Zeeb.
\newblock Greedy sampling using nonlinear optimization.
\newblock In {\em Reduced Order Methods for Modeling and Computational
  Reduction}, pages 137--157. Springer International Publishing, Cham, 2014.

\bibitem{WohlmuthHier}
B.~I. Wohlmuth.
\newblock {Hierarchical a Posteriori Error Estimators for Mortar Finite Element
  Methods with Lagrange Multipliers}.
\newblock {\em SIAM J.\ Numer.\ Anal.}, 36(5):1636--1658, 1999.

\bibitem{MR3318670}
M.~Yano.
\newblock A reduced basis method with exact-solution certificates for steady
  symmetric coercive equations.
\newblock {\em Comput. Methods Appl. Mech. Engrg.}, 287:290--309, 2015.

\bibitem{MR3460105}
M.~Yano.
\newblock A minimum-residual mixed reduced basis method: exact residual
  certification and simultaneous finite-element reduced-basis refinement.
\newblock {\em ESAIM Math.\ Model.\ Numer.\ Anal.}, 50(1):163--185, 2016.

\bibitem{MR703179}
O.~C. Zienkiewicz, D.~W. Kelly, J.~Gago, and I.~Babu{\v{s}}ka.
\newblock Hierarchical finite element approaches, error estimates and adaptive
  refinement.
\newblock In {\em The mathematics of finite elements and applications, {IV}
  ({U}xbridge, 1981)}, pages 313--346. Academic Press, London-New York, 1982.

\bibitem{MR2776914}
Q.~Zou, A.~Veeser, R.~Kornhuber, and C.~Gr\"aser.
\newblock Hierarchical error estimates for the energy functional in obstacle
  problems.
\newblock {\em Numer. Math.}, 117(4):653--677, 2011.

\end{thebibliography}

\end{document}